\documentclass[11pt]{amsart}
\usepackage{url}
\usepackage{graphicx}
\usepackage{amssymb, amscd}
\usepackage{xy}
\usepackage{bm}
\xyoption{all}
\usepackage[linktocpage]{hyperref}
\usepackage[dvipsnames]{xcolor}
\usepackage{outline}
\usepackage{epstopdf}
\usepackage{rotating}
\usepackage{fancyhdr}
\usepackage[left=1in,top=1in,right=1in,bottom=1in,head=.2in]{geometry}
\usepackage{mathabx}

\usepackage{pinlabel}
\newcommand{\fcs}[2]{\,_{#1}\hskip-2.5pt\cs_{#2}}

\renewcommand{\phi}{\varphi}

\newcommand{\la}{\langle}
\newcommand{\ra}{\rangle}

\newcommand{\Z}{\mathbb Z}

\newcommand{\R}{\mathbb R}

\def\spinc{\ifmmode{\textrm{Spin}^c}\else{$\textrm{Spin}^c$}\fi}

\newtheorem{theorem}{Theorem}[section]
\newtheorem{thm}{Theorem}

\newtheorem{lemma}[theorem]{Lemma}

\newtheorem{proposition}[theorem]{Proposition}

\theoremstyle{definition}
\newtheorem{definition}[theorem]{Definition}
\newtheorem{remark}[theorem]{Remark}
\newtheorem{question}[theorem]{Question}

\newcommand{\fig}[3]{\begin{figure}[h!] \includegraphics[height=#1pt]{#2}#3\end{figure}}

\newcommand{\figref}[1]{Figure~\ref{F:#1}}
\newcommand{\secref}[1]{Section~\ref{S:#1}}
\newcommand{\thmref}[1]{Theorem~\ref{T:#1}}
\newcommand{\lemref}[1]{Lemma~\ref{L:#1}}

\newcommand{\bc}{\mathbb C}

\newcommand{\sss}{{S^2\hskip-2pt\times\hskip-2pt S^2}}
\newcommand{\sst}{S^2\widetilde{\times}S^2}
\newcommand{\cptwo}{\bc\textup{P}^2}
\newcommand{\cpone}{\bc\textup{P}^1}

\newcommand{\cs}{\mathbin{\#}} 

\newcommand{\p}{\partial}

\title{Knotted surfaces in $4$-manifolds by Knot surgery and Stabilization}

\author[Hee Jung Kim]{Hee Jung Kim}
\address{Department of Mathematical Sciences\newline\indent
Seoul National University\newline\indent
Seoul 790-784, Korea}
\email{heejungorama@gmail.com}

\thanks{Supported by NRF grant 2015R1D1A1A01059318 and BK21 PLUS SNU Mathematical Sciences Division.
\\ \indent Math.~Subj.~Class.~2010: 57M25 (primary), 57Q60 (secondary)
}

\begin{document}
\begin{abstract} Given a simply-connected closed $4$-manifold $X$ and a smoothly embedded oriented surface $\Sigma$, various constructions based on Fintushel-Stern knot surgery have produced new surfaces in $X$ that are pairwise homeomorphic to $\Sigma$, but not diffeomorphic. We prove that for all known examples of surface knots constructed from knot surgery operations that preserve the fundamental group of the complement of surface knots, they become pairwise diffeomorphic after stabilizing by connected summing with one $\sst$. When $X$ is spin, we show in addition that any surfaces obtained by a knot surgery whose complements have cyclic fundamental group become pairwise diffeomorphic after one stabilization by $\sst$.
\end{abstract}
\maketitle

\vskip-.3in
\vskip-.3in


\section{Introduction}\label{I:intro}

Let $X$ be a smooth closed $4$-manifold and $\Sigma$ be a smoothly embedded surface. An `exotic embedding' of a  surface $\Sigma$ in $X$ is a smooth embedding in $X$ that is pairwise homeomorphic to $\Sigma$, but not diffeomorphic.  
The `stabilization' of given pair $(X,\Sigma)$ is the process of connected summing with a standard manifold pair $(\sss,\emptyset)$ or $(\sst,\emptyset)$, where $\sst$ denotes the non-trivial $S^2$ bundle over $S^2$. 

The recent work~\cite{auckly-kim-melvin-ruberman} of Auckly, Melvin, Ruberman, and the author has constructed the first examples of exotic $2$-spheres in closed simply-connected $4$-manifolds that become pairwise smoothly isotopic after `single' stabilization by $(\sss,\emptyset)$. In this context, one can ask if this stabilization phenomenon arises to exotic surfaces with higher genus.

While a great deal of exotic embeddings in $4$-manifolds are known through various constructions~\cite{finashin-kreck-viro, finashin, fs:rim, kim:twistrim, kim-danny:topotriviality, kim-danny:symsurface, kim-danny:dpsurgery}, interestingly most examples of exotic embeddings for oriented surfaces in simply-connected $4$-manifolds derive from the constructions based on `knot surgery' of Fintushel-Stern~\cite{fs:knotsurgery}. Knot surgery using a knot $K$ in $S^3$ is the operation of removing a neighborhood of a torus $T$ and replacing it by a product of $S^1$ and the exterior of the knot $K$.
Fintushel and Stern provided an effective way to detect the change of diffeomorphism type for knot surgery, showing that the Alexander polynomial of $K$ is reflected in the Seiberg-Witten invariant for a knot surgered $4$-manifold. This allows one to quickly construct and detect infinite families of exotic smooth structures on a large class of $4$-manifolds. Likewise, knot surgery can be used to change a smooth structure of smoothly embedded surface in a $4$-manifold. This approach relies on `ambient surgery' whereby a given surface $\Sigma$ is surgered to a new surface $\Sigma_K(\varphi)$, leaving the ambient manifold $X$ fixed.
The rim surgery of Fintushel-Stern~\cite{fs:rim}, author's twist rim surgery~\cite{kim:twistrim}, and Finashin's annulus rim surgery~\cite{finashin} are examples of this technique, underlying most examples of smoothly knotted oriented surfaces in a simply-connected closed $4$-manifold.

In the direction of the study of stabilization for exotic smooth structures, Auckly~\cite{auckly:stable} for $\sst$ and Akbulut~\cite{akbulut:fs-knot-surgery} for $\sss$ proved that a simply-connected $4$-manifold $X$ and its knot surgered manifold $X_K(\varphi)$ become diffeomorphic after single stabilization by $\sst$ or $\sss$, referred to as {\em $1$-stably equivalent} with the terminology in~\cite{auckly-kim-melvin-ruberman}; see~\cite{baykur-sunukjian:round} for the alternative proof.

This paper investigates the analogous stabilization question for knotted surfaces produced by all of the known constructions based on knot surgery i.e. rim surgery, twist rim surgery, and annulus rim surgery.

The Wall's stable $h$-cobordism theorem~\cite{wall:4-manifolds} states that homotopy equivalent, simply-connected $4$-manifolds become diffeomorphic after stabilization by some finite number of $\sss$ or $\sst$. It also holds for embedded surfaces (up to diffeomorphism of pairs) with simply connected complements in a $4$-manifold that represent the same homology class~\cite{quinn:stabtopology}. And, in fact, all known examples need only one stabilization to be diffeomorphic. 
So, the stabilization question for a knot surgered pair $(X,\Sigma_K(\varphi))$ would be the following. In this paper, we will use the terminology `surface knot group' for the fundamental group of surface complement in a $4$-manifold.
\begin{question}
Suppose that $X$ is simply connected and $\Sigma$ is an oriented smoothly embedded surface. Let $(X,\Sigma_{K}(\varphi))$ be a pair obtained by a knot surgery from $(X,\Sigma)$. If $\Sigma_{K}(\varphi)$ and $\Sigma$ have the same surface knot group in $X$ then are they $1$-stably equivalent?
\end{question}
This paper answers this question affirmatively for all of the currently known constructions. The precise statements are given in~\secref{mainthm} (Theorems~\ref{T:mainthmA},~\ref{T:mainthmB},~\ref{T:mainthmC}) after we discuss the known techniques for constructing exotic surfaces. 

Rim surgery of Fintushel and Stern~\cite{fs:rim} constructed an infinite family of exotic smooth embedding for surfaces with simply-connected complements in a simply-connected $4$-manifold. Finashin used annulus rim surgery~\cite{finashin} for knotting algebraic curves in $\cptwo$, and produced surfaces that are smoothly not isotopic to algebraic curves for degree $d\ge 5$, but the topological classification of his examples was open. The later work~\cite{kim:twistrim} of the author introduced a method, called twist rim surgery, of knotting surfaces that produced exotic embeddings for surfaces with cyclic knot groups in a simply-connected $4$-manifold. Applied to algebraic curves in $\cptwo$, the twist rim surgery leads to the construction of infinitely many exotic smooth structures on algebraic curves of degree $d\ge 3$. For degrees $1$ and $2$, the surfaces are spheres, and it is not easy to distinguish these by Seiberg-Witten invariants. The work of Ruberman and author~\cite{kim-danny:topotriviality} strengthened the criterion from~\cite{kim:twistrim} for topological equivalence of surfaces by showing that any surfaces produced by a knot surgery that preserve a cyclic knot group is topologically standard. As a consequence, we deduced that Finashin's examples are topologically standard. Despite some results about the existence of symplectic, noncomplex surfaces as well as smooth surfaces without symplectic structures, the main classical source of examples for smooth embeddings codimension $2$ had been complex curves. A subsequent work~\cite{kim-danny:symsurface} extended Gompf's theorem about the fundamental group of symplectic manifolds to the relative case, showing that any finitely presented group can be realized as the fundamental group of complement of a symplectic surface in a simply-connected symplectic $4$-manifold, whereas the fundamental groups of complement of complex curves are quite restricted. Those examples can be further smoothly knotted by twist rim surgery so that it has led to a large class of exotic embeddings. 
Another interesting aspect of twist rim surgery is that some iteration of the twist rim surgery gives a way of constructing new smooth surfaces with certain non-abelian finite surface knot group. One consequence is that it gave an infinite family of exotic surfaces in $S^2\times S^2$ with knot group a dihedral group $D_{2p}$, for any odd $p$. 

In this paper, we prove that for all known examples of surface knots constructed from rim surgery, twisted rim surgery, and annulus rim surgery that preserve their surface knot groups, they become pairwise diffeomorphic after `single' stabilization by $(\sst,\emptyset)$. 

Another result includes an interesting phenomenon in the relative version of stabilization i.e. connected sum with $(\sss,\emptyset)$ or $(\sst,\emptyset)$. It is known that for a nonspin simply-connected $4$-manifold $X$, essentially due to Wall~\cite{wall:diffeomorphisms}, $X\cs \sss$ is diffeomorphic to $X\cs\sst$, but surprisingly it is not true for the relative case. Theorem~\ref{T:mainthmD} proves that for a degree $d$-curve $\Sigma_d$ in $\cptwo$, $(\cptwo\cs\sss,\Sigma_d)$ is {\em not} even pairwise homeomorphic to $(\cptwo\cs\sst,\Sigma_d)$, even when $\cptwo-\Sigma_d$ is not spin i.e. $d=$ even. 

Finally, we show that if a knot surgery $(X,\Sigma)\to(X_K(\varphi),\Sigma_K(\varphi))$ is cyclic, which is defined to be a surgery preserving a cyclic surface knot group~\cite{kim-danny:topotriviality}, then the pairs are $1$-stably equivalent by connected summing with $(\sst,\emptyset)$ in the case that $X$ is spin. 

\begin{remark}
Note that here we will not impose any extra assumptions on $\Sigma$ other than that $\Sigma$ is an oriented smoothly embedded surface in a simply-connected closed $4$-manifold $X$. Recall that the constructions of (twisted) rim surgery and annulus rim surgery can provide exotic embeddings of $\Sigma$ when $\Sigma$ is a surface of positive genus and $(X,\Sigma)$ has a non-trivial relative Seiberg-Witten invariant~\cite{fs:knotsurgery,fs:rim,fs1:rim,sunukjian:polyisotopy} (or a relative Heegaard-Floer invariant as in the version of Mark~\cite{mark}).

\end{remark}
The main theorems are precisely stated in the next section where it carefully describes when surface knot groups are preserved for each knotting construction. And it includes the proof of Theorem~\ref{T:mainthmD}. 


\section{Main Theorems}\label{S:mainthm}
Before we state our results, the notions of `equivalence' of embeddings of surfaces in a $4\text{-manifold}$ should be clarified as in~\cite{auckly-kim-melvin-ruberman}:

\begin{definition}\label{D:equivalencesurfaces}
Two smoothly embedded surfaces $\Sigma\,,\,\Sigma' $ in a smooth $4$-manifold $Z$ are {\em equivalent} if there is an orientation preserving pairwise diffeomorphism of $(Z, \Sigma)$ to $(Z, \Sigma')$. Two smoothly embedded surfaces $\Sigma\,,\,\Sigma' $ in a smooth $4$-manifold $X$ are {\em  $n$-stably equivalent} if the natural embeddings $\Sigma\,,\,\Sigma' \subset X\cs n\hskip.5pt\sss$ (or $n\hskip.5pt\sst$) are equivalent in $X\cs n\hskip.5pt\sss$ (or $n\hskip.5pt\sst$), but not in $X\cs k\hskip.5pt\sss$ (or $k\hskip.5pt\sst$) for any $k\le n-1$. 
\end{definition}

Note that our constructed exotic $2$-spheres in~\cite{auckly-kim-melvin-ruberman} have simply-connected complements and they are $1$-stably {\em isotopic} which is a stronger notion of equivalence of surfaces. It is still open to see the distinction between equivalence of surfaces up to {\em diffeomorphism} and {\em smooth isotopy}~\cite{ruberman:isotopy,ruberman:polyisotopy}, while this issue does not arise in the topological case~\cite{perron:isotopy,quinn:isotopy}. Here our stabilization by $\sss$ (or $\sst$) is taken in the `outside' of embedded surfaces in $X$, but there is another notion of stabilization for embedded surfaces, adding an unknotted handle to the surface. The work~\cite{baykur-sunukjian:surfacestabilization} of Baykur-Sunukjian showed that all constructions of exotic knotting of surfaces produce surfaces that become smoothly {\em isotopic} after adding a single handle in a standard way. 
\vskip5pt

Let $X$ be a smooth $4$-manifold containing a torus $T$ with a trivial normal bundle and let $K$ be a knot in $S^3$ with its closed complement $E(K)$. Fintushel-Stern's knot surgery ~\cite{fs:knotsurgery} is the process of removing a neighborhood of $T$ from $X$ and re-gluing $S^1\times E(K)$ via a diffeomorphism $\varphi$ on the boundary to form $X_{K}(\varphi)=X-\nu(T)\cup_{\varphi} S^1\times E(K)$. Denote by $\mu_T$ the boundary of the normal disk of $T$, and let the meridian/longitude of $K$ be $\mu_K$ and $\lambda_K$ respectively. Here the gluing map $\varphi:\p\nu(T)\to S^1\times \p E(K)$ can be chosen by any diffeomorphism such that $\varphi_*\mu_T=\lambda_K$. When $X$ is a simply-connected closed $4$-manifold, this operation doesn't change the homeomorphism type, while it may change its diffeomorphism type. 

Applied to a torus in the exterior of an embedded surface in a closed $4$-manifold, the knot surgery can change embeddings of surfaces in $4$-manifolds.
We assume that $X$ is a smooth simply-connected closed $4$-manifold, and $\Sigma$ is an oriented embedded surface in $X$ throughout the paper. Then the fundamental group $\pi_1(X-\Sigma)$ is normally generated by a meridian  $\mu_\Sigma$ of surface. 
For a surface $\Sigma$ carrying a non-trivial homology class in $X$, the first homology group $H_1(X-\Sigma)$ is always finite cyclic, of order that we will usually write as $d$.  The process of knotting an embedded surface $\Sigma$ in $X$ can be obtained by performing knot surgery on a torus in the exterior $X-\nu(\Sigma)$, and then gluing $(X-\nu(\Sigma))_{K}(\varphi)$ back in the neighborhood of the surface $\nu(\Sigma)$ gives a new embedding of $\Sigma$ in $X_{K}(\varphi)$ with image $\Sigma_{K}(\varphi)$. In the case of rim surgery, twist rim surgery, and annulus rim surgery, there is a canonical identification between $X$ and $X_{K}(\varphi)$ so that we can view $\Sigma_{K}(\varphi)$ as an embedding in $X$; see~\secref{prelim} for more details of these constructions. In general the resulting homeomorphism/diffeomorphism type of the new embedding $\Sigma_{K}(\phi)$ depends on a choice of torus $T$, knot $K$, and gluing map $\varphi$. Our results show that the surfaces $\Sigma_{K}(\varphi)$ and $\Sigma$ are $1$-stably equivalent under some circumstances as follows.
\vspace{0.5pc}

Rim surgery deals with surfaces with simply-connected complements in a simply-connected $4$-manifold and doesn't change the fundamental group, so the surface $\Sigma_{K}(\varphi)$ is in fact topologically isotopic to $\Sigma$ by the works in~\cite{perron:isotopy,quinn:isotopy}. The following theorem shows the stabilization result for these surfaces. 

\begin{thm}\label{T:mainthmA} Suppose that $X$ is a simply-connected closed $4$-manifold and $\Sigma$ is an smoothly embedded oriented surface with $\pi_1(X-\Sigma)=1$. Let $(X,\Sigma_{K}(\varphi))$ be a pair obtained by a rim surgery. Then  $(X,\Sigma)\cs (\sst,\emptyset)$ is pairwise diffeomorphic to $(X,\Sigma_{K}(\varphi))\cs (\sst,\emptyset)$.
\end{thm}
\begin{remark}
It turns out that the general $1$-stable {\em isotopy} principle holds for surfaces with simply-connected complements. The recent paper~\cite{auckly-kim-melvin-ruberman-schwartz} of Auckly, Melvin, Ruberman, Schwartz, and the author has shown using Gabai's result~\cite{gabai} that any two homologous surfaces of the same genus embedded in a $4$-manifold $X$ with simply-connected complements are smoothly {\em isotopic} after single stabilization with $\sss$ if the surfaces are ordinary, and $\sst$ if they are characteristic. 

\end{remark}

Finashin's annulus rim surgery~\cite{finashin} requires a suitable annulus $M \cong S^1\times I$ in $X$ to produce a new surface via knotting $\Sigma$ along the annulus. This surgery in his paper is given by an explicit geometric description of the surgered surface, but in~\cite{kim-danny:topotriviality} a knot surgery description for this surgery is provided; see~\secref{annulusrimsurgery} for this description. It is shown in~\cite{finashin,kim-danny:topotriviality} that annulus rim surgery preserves the surface knot group when $\pi_1(X-\Sigma)=\Z_d$, and it turns out that the surface $\Sigma_K(\varphi)$ is topologically isotopic to $\Sigma$ by the work in~\cite[Theorem 1.3]{kim-danny:topotriviality}.

\begin{thm}\label{T:mainthmB} Suppose that $X$ is a simply-connected closed $4$-manifold and $\Sigma$ is an smoothly embedded oriented surface with $\pi_1(X-\Sigma)=\Z_d$. Let $(X,\Sigma_{K}(\varphi))$ be a pair obtained by an annulus rim surgery. Then $(X,\Sigma)\cs (\sst,\emptyset)$ is pairwise diffeomorphic to $(X,\Sigma_{K}(\varphi))\cs (\sst,\emptyset)$.
\end{thm}

In order to explore this phenomenon for surface knots with arbitrary knot groups, we consider twist rim surgery~\cite{kim:twistrim,kim-danny:topotriviality,kim-danny:symsurface}, a variation of the Finstushel-Stern's rim surgery with additional twists parallel to a meridian and a longitude of a knot $K$. We write the meridian twist rim surgery as `$m$-twist rim surgery' when we wish to indicate the number of twists applied on the meridian of $K$, and also denote by $\Sigma_K(m)$ the new embedding produced from $\Sigma$ under the surgery. The way in which $m$-twist rim surgery affects the fundamental group of a surface knot depends to some degree on the relation between $m$ and $d$, where $H_1(X-\Sigma)\cong Z_d$.
For example, when $m=\pm 1$, the twist rim surgery always preserves the fundamental group of a surface knot; the proof was given for $1$-twist in~\cite[Proposition 2.3]{kim-danny:symsurface}, but it works for $-1$-twist in the exactly same way. The $1$-twist rim surgery allows us to construct exotic smooth embeddings for a symplectic surface with any finitely presented knot group in a symplectic $4$-manifold (see~\cite[Theorem 3.1, 5.2]{kim-danny:symsurface} for more details).  More generally, Proposition 2.4 in~\cite{kim-danny:symsurface} shows when an $m$-twist rim surgery preserves the fundamental group of surface knots as shown that for a surface $\Sigma\subset X$ with $H_1(X-\Sigma)\cong \Z_d$, if $(m,d)=1$ and the meridian $\mu_{\Sigma}$ has order $d$ in $\pi_1(X-\Sigma)$, then $\pi_1(X-\Sigma)\cong \pi_1 (X- \Sigma_{K}(m))$. This criterion is used to produce infinitely many exotic embeddings in $S^2\times S^2$ with knot group a dihedral group $D_{2p}$ for any odd $p$~\cite[Theorem 5.1]{kim-danny:symsurface}. Note that when $\pi_1(X-\Sigma)=\Z_d$, the surface $\Sigma_{K}(m)$ with $(m,d)=1$ is topologically isotopic to $\Sigma$; see~\cite[Theorem 1.3]{kim-danny:topotriviality}. But for arbitrary surface knot groups, when the knot $K$ is chosen carefully, $\Sigma_{K}(m)$ is equivalent to $\Sigma$ up to smooth $s$-cobordism; see~\cite{kim-danny:symsurface} for more details. In all cases that surface knot groups are preserved under twist rim surgery, we show that $\Sigma_{K}(m)$ and $\Sigma$ are $1$-stably equivalent:

\begin{thm}\label{T:mainthmC} 
Suppose that the surface $\Sigma\subset X$ has $H_1(X-\Sigma)\cong \Z_d$, and let $\pi_1(X-\Sigma)$ be any group $G$. Then the following is true.
\vskip4pt
\begin{enumerate}
\item\label{t:sigma1} $(X\cs\sst,\Sigma)$ is pairwise diffeomorphic to $(X\cs\sst, \Sigma_{K}(\pm 1))$. 
\vskip4pt
\item\label{t:sigmam} If $(m,d)=1$ and $\mu_{\Sigma}$ has order $d$ in $\pi_1(X-\Sigma)$ then $(X\cs\sst,\Sigma)$ is pairwise diffeomorphic to $(X\cs\sst, \Sigma_{K}(m))$. 
\end{enumerate}

\end{thm}
\vskip5pt

Now, we give a simple proof to show an interesting phenomenon in this relative stabilization.  Wall's stabilization result~\cite{wall:diffeomorphisms} for a nonspin simply-connected $4$-manifold $X$ shows that $X\cs \sss$ is diffeomorphic to $X\cs\sst$, but interestingly it fails as follows: 

\begin{thm}\label{T:mainthmD} Let $\Sigma_d$ be a degree $d$-curve in $\cptwo$. Then the pair $(\cptwo\cs\sss,\Sigma_d)$ is `not' pairwise homeomorphic to $(\cptwo\cs\sst,\Sigma_d)$. 
\end{thm}

\begin{proof}
If $d$ is odd then it is obvious since $\cptwo-\Sigma_d$ is spin. But, we will show that the pairs are still not homeomorphic in the case that $d$ is even so that $\cptwo-\Sigma_d$ is not spin. We claim that there is no odd class in $H_2((\cptwo-\Sigma_d)\cs\sss)$. For any $[S]\in H_2((\cptwo-\Sigma_d)\cs\sss)$, the homology class $[S]$ can be written by $k[\cpone]+n_a S_a+n_b S_b$ in $H_2(\cptwo\cs\sss)$, where $S_a$ and $S_b$ denote the first and second generators of $H_2(\sss)$ respectively. Then it gives $[S]\cdot[\Sigma_d]=kd$ that must be zero, so $k=0$. This implies that there is no odd class in $H_2((\cptwo-\Sigma_d)\cs\sss)$, but there is in $H_2((\cptwo-\Sigma_d)\cs\sst)$. 
\end{proof}

\begin{remark}
It is worth pointing out that there is no odd class in $\cptwo-\Sigma_d$ even when $\cptwo-\Sigma_d$ is not spin. To understand this, first note that the nonzero element $\alpha\in H_2(\cptwo-\Sigma_d)$ has $\alpha^2=0$ as shown in the above proof, so it gives $Q_{\cptwo-\Sigma_d}=0$. In fact the handlebody picture of $\cptwo-\Sigma_d$ shows that there are $2g$ $0$-framed $2$-handles, and a $1$-framed $2$-handle which is $d$-times linked with a $1$-handle; see Exercises 6.2.12.(c)~\cite{GS}. For $d=$ even, there is a $\Z_2$-homology class $\beta$ of the $1$-framed $2$-handle and over $\Z_2$, the intersection form is given by $[1]$. By the Wu formula, $w_2(\cptwo-\Sigma_d)$ vanishes on $\alpha$, but has value $1$ on the $\Z_2$-homology class $\beta$.
\end{remark}

Finally, we focus on the case that $\pi_1(X-\Sigma)$ is a cyclic group $\Z_d$, and investigate the stabilization problem of knot surgery. As the terminology in~\cite{kim-danny:topotriviality}, if a knot surgery $(X,\Sigma)\to (X_{K}(\varphi), \Sigma_{K}(\varphi))$ satisfies that $\pi_1(X-\Sigma)\cong \pi_1 (X_{K}(\varphi)- \Sigma_{K}(\varphi))$ is cyclic then the knot surgery is called a {\em cyclic surgery}. In~\cite[Theorem 1.2]{kim-danny:topotriviality}, Ruberman and the author showed that for any pair $(X,\Sigma)$ that $X$ is simply-connected and $\Sigma$ is an embedded surface with $\pi_1(X-\Sigma)\cong \Z_d$, if a knot surgery $(X,\Sigma)\to (X_{K}(\varphi), \Sigma_{K}(\varphi))$ is cyclic then there is a pairwise homeomorphism $(X,\Sigma)\to (X_{K}(\varphi), \Sigma_{K}(\varphi))$. Thus, it is natural to ask the $1$-stable equivalence for the cyclic knot surgery. We answer for this question in the case that $X$ is spin:

 \begin{thm}\label{T:mainthmE} Let $X$ be a simply-connected, closed, spin $4$-manifold and $\Sigma$ be an embedded oriented surface with $\pi_1(X-\Sigma)\cong \Z_d$. Suppose that the knot surgery $(X,\Sigma)\to (X_{K}(\varphi), \Sigma_{K}(\varphi))$ is cyclic. Then $(X,\Sigma)$ is pairwise diffeomorphic to $(X_{K}(\varphi), \Sigma_{K}(\varphi))$ after one stabilization with $(\sst,\emptyset)$.
\end{thm}

\begin{remark}
In the contrast to the well-known stabilization theorems for simply-connected $4$-manifolds, the relative stabilization for cyclic knot surgery doesn't seem to give any general statement for a choice of $(\sss,\emptyset)$ or $(\sst,\emptyset)$ in the case of a nonspin $4$-manifold $X$. Our main argument for stabilization results will rely on proving the $1$-stable equivalence of surface knots $\Sigma_{K_{i}}(\varphi)$ and $\Sigma_{K_{i+1}}(\varphi)$ for two knots $K_{i}$ and $K_{i+1}$ related by one crossing change.  At each stage of crossing change that will make any knot to an unknot, one cannot assert that the pairs $(X_{K_{i}}(\varphi),\Sigma_{K_{i}}(\varphi))$ and $(X_{K_{i+1}}(\varphi),\Sigma_{K_{i+1}}(\varphi))$ become pairwise diffeomorphic after one stabilization with {\em only} $(\sss,\emptyset)$ or with {\em only} $(\sst,\emptyset)$ when $X-\Sigma$ is nonspin. This issue arises because $(X_{K_{i}}(\varphi),\Sigma_{K_{i}}(\varphi))\cs (\sss,\emptyset)$ may not be pairwise diffeomorphic to $(X_{K_{i}}(\varphi),\Sigma_{K_{i}}(\varphi))\cs (\sst,\emptyset)$ as seen in the proof of~\thmref{mainthmD}.

\end{remark}


\section{Knot surgery constructions to change embeddings in $4$-manifolds}\label{S:prelim}

Let $X$ be a simply-connected closed $4$-manifold and $\Sigma$ be an embedded oriented surface.


\subsection{Twist rim surgery}\label{S:twistedrimsurgery}
Let $R_{\alpha}$ be a torus with $R_{\alpha}\cdot R_{\alpha}=0$ (called a {\em rim torus}) that is the preimage in $\p{\nu(\Sigma)}$ of a closed curve $\alpha\subset\Sigma$. Identify the neighborhood $\nu(\alpha)$ of the curve $\alpha$ in $X$ with $S^1\times I\times D^2=S^1\times B^3$ where $\nu(\alpha)$ in $\Sigma$ is $S^1\times I$. In this trivialization, let $\beta$ be a pushed-in copy of the meridian circle $\{0\}\times\p D^2\subset I\times D^2$, so it is isotopic to a meridian of $\Sigma$. Then the rim torus $R_{\alpha}$ can be written as $\alpha\times\beta\subset S^1\times(B^3, I)$ and we will identify a neighborhood $\nu(R_{\alpha})$ of $R_{\alpha}$ with $\alpha\times (\beta\times D^2)\subset S^1\times(B^3, I)$. Let $K$ be a knot in $S^3$ with its closed exterior $E(K)$, and $\mu_K, \lambda_K$ denotes a pair of meridian-longitude of $K$. The $m$-twists and $n$-rolls of rim surgery on $(X,\Sigma)$ is defined by 
$$
(X,\Sigma_{K}(\varphi))=(X,\Sigma)-\nu(R_{\alpha})\cup_{\varphi}S^1\times E(K).
$$

Here the gluing map $\varphi:  \p \nu(R_{\alpha})\to S^1\times \p E(K) $ is the diffeomorphism determined by 
\begin{equation}\label{twistrimsurgery:gluing}
\varphi_{*}(\alpha')= m\mu_{K}+n\lambda_{K}+[S^1], \quad  
\varphi_{*}(\beta')=\mu_{K},  \quad   \text{and}  \quad  \varphi_{*}(\mu_{R})=\lambda_{K}
\end{equation}
with respect to a basis $\{\alpha',\beta', \mu_{R} \}$ for $H_1(\p{\nu(R_{\alpha})})$ and $\{[S^1],\mu_K, \lambda_K\}$ for $H_1(S^1\times \p E(K)) $, where $\alpha',\beta'$ are the pushoffs of  $\alpha$, $\beta$ into $\p \nu(R_{\alpha})$ and $\mu_{R}$ denotes a meridian of the rim torus. 

Such a gluing corresponds to the \emph{spinning} construction of the rim surgery of Fintushel-Stern i.e. $m=n=0$, adding a combination of $m$-fold twist spinning~\cite{zeeman:twist} and $n$-fold roll spinning~\cite{fox:rolling,litherland:deform}. 
It is useful to specify these twists by classical diffeomorphisms that give equivalent descriptions for the twisted rim surgery.

Consider self-diffeomorphisms denoted by $\tau$ and $\rho$ of $(S^3,K)$ that correspond to twists parallel to a meridian and a longitude of $K$ respectively.
Let $\partial E(K)\times I=K\times \partial D^2\times I$ be a collar of $\partial E(K)$ in $E(K)$ under a suitable trivialization with $0$-framing. Identify $K$ with $S^1\cong \R/\Z$ and then the twist map $\tau$ is given by
\begin{equation}\label{tau}
\tau(\overline\theta, e^{i\psi}, t) =(\overline\theta,
e^{i(\psi + 2\pi t)}, t) \quad \mbox{for} \quad (\overline\theta,
e^{i\psi}, t)\in K\times{\partial{D^2}}\times{I}
\end{equation}
and otherwise, $\tau(y)=y$. 

Similarly, a roll, $\rho$, is obtained from
$\rho(\overline\theta, e^{i\psi}, t) =(\overline{\theta+t}, e^{i\psi}, t)$
by extending as the identity on the rest of $(S^3,K)$. 

Although a roll can also produce exotic embeddings, we will only deal with an $m$-twist rim surgery in this paper since a meridian twist is sufficiently useful to construct all desired smoothly knotted surfaces. Most of the arguments for the stabilization result of the $m$-twist rim surgery be easily modified to address the rolling as well. 

Writing $(S^3,K) = (B^3,K_+) \cup  (B^3,K_{-})$ where $(B^3,K_{-})$ is an unknotted ball pair, we regard $\tau$ as an automorphism of $(B^3,K_+)$. 
Since the rim torus $R_{\alpha}$ lies in a neighborhood of the curve $\alpha$, the twisted rim surgery performed in $\nu(\alpha)\cong S^1\times (B^3,I)$ gives rise to the mapping torus of $(B^3,K_+)$ with monodromy given by the twist map $\tau$. So the $m$-twisted rim surgery on $(X,\Sigma)$ can be written as follows;
\begin{equation}\label{twistrimsurgery2}
(X,\Sigma_{K}(m))=(X,\Sigma)-S^1\times (B^3,I)\cup_{\partial} S^1 \times_{\tau^m}(B^3,K_+).
\end{equation}

In doing any rim surgery (twisted or otherwise) we assume that $\alpha \subset \Sigma$ is a curve for which there is a framing of $\nu(\Sigma)$ along $\alpha$ such that the pushoff of $\alpha$ into $\partial\nu(\Sigma)$ is null-homotopic in $X-\Sigma$.  
But we don't assume that $\alpha$ is a non-separating curve on $\Sigma$, which is necessary to distinguish the diffeomorphism type of $\Sigma_{K}(m)$ from that of $\Sigma$ with Seiberg-Witten invariant. Note from~\cite[Lemma 2.2]{kim:twistrim} that if $\alpha$ bounds a disk in $\Sigma$, the surface $\Sigma_{K}(m)$ is the connected sum of $\Sigma$ with the $m$-twist spun knot $K(m)$ of Zeeman~\cite{zeeman:twist}. Our stabilization results include this example as well. 

\subsection{\bf Twisted rim surgery and the surface knot group}\label{S:twistrimfundamentalgp}

As mentioned in~\secref{mainthm}, $\pm 1$-twist rim surgery always preserves surface knot groups~\cite[Proposition 2.3]{kim-danny:symsurface}, and also Proposition 2.4 in~\cite{kim-danny:symsurface} shows when an $m$-twist rim surgery preserves the fundamental group.  Here we will revisit Proposition 2.4 with more elementary argument (compare the proof in~\cite{kim-danny:symsurface}) since it explicitly provides the presentation of the fundamental group for later use. 

\begin{proposition}[Proposition 2.4 in~\cite{kim-danny:symsurface}]\label{P:twistrimfundagp}
Let $\pi_1(X -\Sigma)$ be any group $G$. Suppose that the surface $\Sigma \subset X$ has $H_1(X -\Sigma) =\Z_d$ and the meridian $\mu_\Sigma$ has order $d$ in $\pi_1(X -\Sigma)$.  If $(m,d)=1$ then $\pi_1(X-\Sigma_{K}(m))$ is isomorphic to $G$.
\end{proposition}

\begin{proof}

In order to investigate a presentation of $\pi_1(X-\Sigma_{K}(m))$, we first consider the decomposition of $X-\Sigma_K(m)$ induced from~\eqref{twistrimsurgery2}:
\begin{equation}\label{cptwistrimsurgery2}
X-\Sigma_K(m)=X-\Sigma-S^1\times (B^3,I)\cup_{\partial} S^1 \times_{\tau^m}(B^3-K_+).
\end{equation}

Choosing a base point $*$ at the intersection of two components in this decomposition~\eqref{cptwistrimsurgery2}, we get the following diagram from the van Kampen theorem;
\begin{equation}\label{vandia}
\xymatrix@C=-2pc{
&\pi_1(X-\Sigma-S^1\times (B^3, I)) \ar[dr]^{j_1}&\\
{\pi_1 (S^1\times(\partial{B^3}-\{\mbox{two points}\}))}  \ar[ur]^{i_1} \ar[dr]^{i_2}
&&  \pi_1(X - \Sigma_{K}(m))\\
& \pi_1(S^1\times_{\tau^m}(B^3-K_+)) \ar[ur]^{j_2}&
}
\end{equation}

In the diagram, each map is obviously induced by an inclusion and $\pi_1 (S^1\times(\partial{B^3}-\{\mbox{two points}
\}))$ is generated by two elements $[S^1]$ and $\mu$. So, the relations in a presentation of $\pi_1(X - \Sigma_{K}(m))$ are given by $i_1[S^1]=\alpha'$ which is trivial by the assumption that the pushoff $\alpha'$ of $\alpha$ is null-homotopic in $X-\Sigma$, and $i_1(\mu)$ is a meridian $\mu_\Sigma$ of $\Sigma$ in $\pi_1(X-\Sigma-S^1\times (B^3,I))\cong\pi_1(X-\Sigma)$. So, it leads the presentation for $\pi_1(X - \Sigma_{K}(m))$ as follows:
\begin{equation}\label{mtwistfundgpeq1}
 \la \pi_1(X-\Sigma) *\pi_1(S^1\times_{\tau^m}
(B^3-K_+)) \mid 1=\delta, \mu_\Sigma=\mu_K \ra ,
\end{equation}
where $\delta=[S^1]$, $\mu_K=[\mu_K]$ denote generators of $\pi_1(S^1\times_{\tau^m}
(B^3-K_+), *)$ in~\figref{fundamentalgp}.
Associated with the relations, this presentation becomes the following:
\begin{equation}\label{mtwistfundgpeq2}
 \la \pi_1(X-\Sigma) *\pi_1(B^3-K_+) \mid \mu_\Sigma=\mu_K, \mu_K^{-m}g\mu_K^{m}=g\  , \forall g\in\pi_1(B^3-K_+) \ra .
\end{equation}
If $(m,d)=1$ and $\mu_{\Sigma}^d=1$, it obviously gives $\pi_1(X-\Sigma)$.
\fig{120}{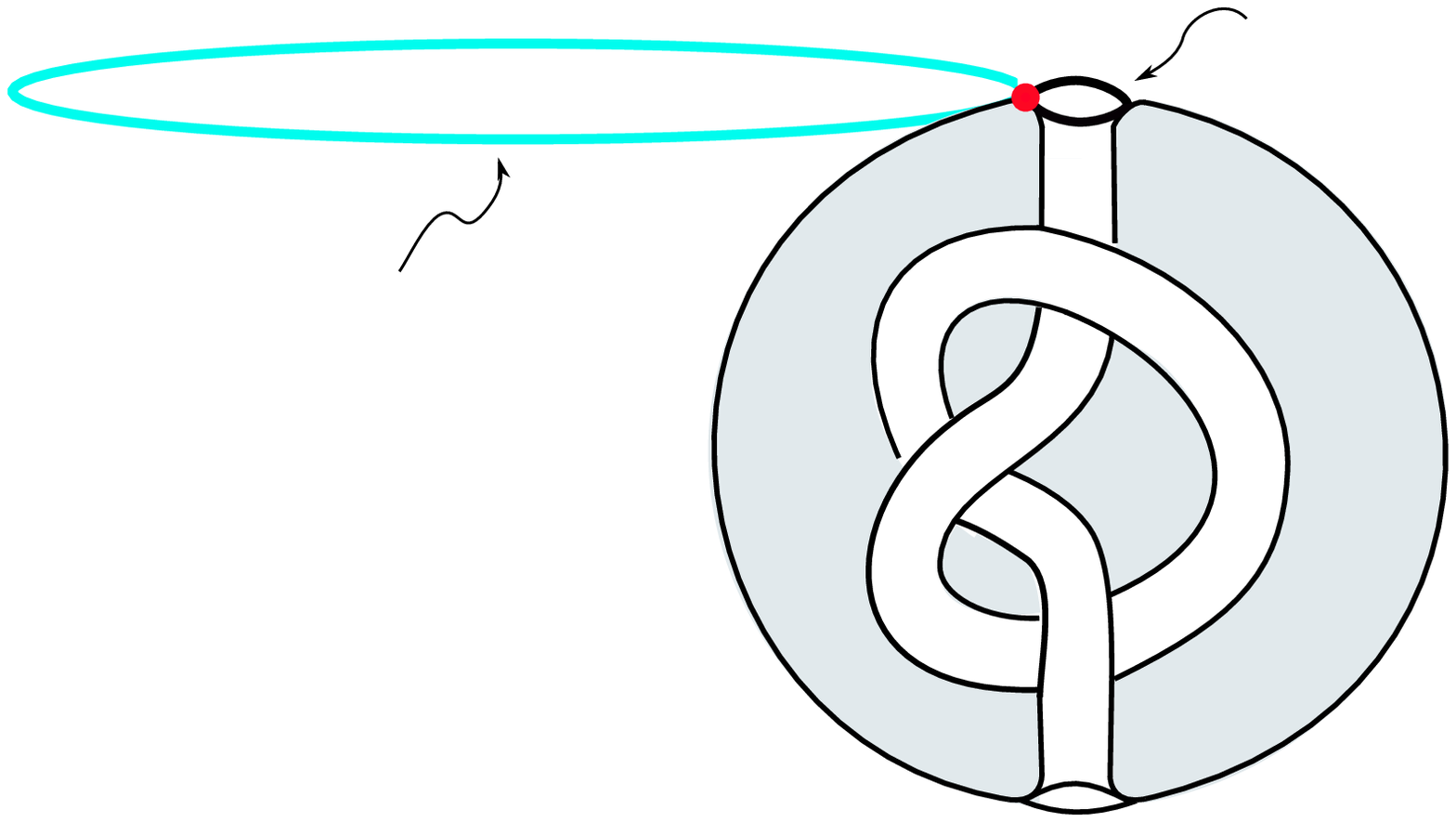}{
 \put(-60,110){\footnotesize$\color{red}{*}$}
 \put(-180,65){\footnotesize$\delta=[S^1,*]$}
   \put(-20,120){\footnotesize$\mu_K=[\mu_K,*]$}
\caption{Generators of $\pi_1(S^1\times_{\tau^m}(B^3-K_+), *)$}
\label{F:fundamentalgp}}


\end{proof}

\begin{remark}
Note that in the diagram, if the image of $\pi_1(S^1\times_{\tau^m} (B^3-K_+))$ in $\pi_1(X - \Sigma_{K}(m))$ is a cyclic subgroup generated by the meridian of $\Sigma_{K}(m)$, then the presentation~\eqref{mtwistfundgpeq1} for $ \pi_1(X - \Sigma_{K}(m))$ readily leads to the group $ \pi_1(X - \Sigma)$. The choice of the parameter $m$ in Proposition~\ref{P:twistrimfundagp} makes this case. This property enables one to see that the stabilization result of Baykur-Sunukjian in~\cite{baykur-sunukjian:surfacestabilization} can be extended to knotted surfaces produced by the $m$-twist rim surgery with the choice of $m$ in Proposition~\ref{P:twistrimfundagp}. They showed that for $m=1$ or any $m$ with $(m,d)=1$ in the case of $\pi_1(X-\Sigma)=\Z_d$, the $m$-twist surgered surfaces become smoothly isotopic by adding one trivial handle. The same argument in~\cite[Section 3.2]{baykur-sunukjian:surfacestabilization} can work for the examples in Proposition~\ref{P:twistrimfundagp} by adding a $1$-handle $h$ at a crossing of the knotted arc $K_+$ in $S^1\times_{\tau^m} (B^3, K_+)$ to unknot it crossing by crossing, where it must be checked that the attached handle is trivial at each stage.  It follows from the work of Boyle~\cite[Theorem 2, Section 2]{Boyle}, showing that surfaces obtained by attaching two different $1$-handles to a surface $F$ are equivalent in  $S^4$ up to isotopy if and only if their `double cosets' of the peripheral subgroup $P$, the image of $\pi_1(\p\nu(F))$ under the inclusion $i_*: \pi_1(\p\nu(F))\to  \pi_1(S^4-\nu(F))$, should be same. Here the double coset induced by a $1$-handle $h$ of the peripheral subgroup $P$ is defined by $P[\alpha*h^c*\beta^{-1}]P$, where $h^c$ is a core of the $1$-handle in the exterior of $\nu(F)$ connecting a point $b_0$ to $b_1$ in $\p\nu(F)$ and $\alpha$, $\beta$ are paths in $\p\nu(F)$ starting from the base point of $\pi_1(\p\nu(F))$ while $\alpha$ ends at $b_0$ and $\beta$ ends at $b_1$.  So this makes $[\alpha*h^c*\beta^{-1}]$ an element in $\pi_1(S^4-\nu(F))$, and from his result~\cite[Theorem 2, Section 2]{Boyle} it can be shown that a $1$-handle $h$ on $F$ is trivial if and only if its double coset $P[\alpha*h^c*\beta^{-1}]P=P$; see Corollary 3~\cite[Section 3]{Boyle}. Note that this works for general $4$-manifold as shown in~\cite[Lemma 3]{baykur-sunukjian:surfacestabilization}. In our circumstance, the image of $\pi_1(S^1\times_{\tau^m} (B^3-K_+))$ in $\pi_1(X - \Sigma_{K}(m))$ is a cyclic subgroup generated by the meridian of $\Sigma_{K}(m)$ so that every handle $h$ attached within $S^1\times_{\tau^m} B^3$ is homotopic to a handle attached along some number of meridians to the surface $\Sigma_{K}(m)$. In other words, the core $h^c$ of every attached handle can be isotoped to lie in $\p\nu(\Sigma_{K}(m))$, and so such cores represent a trivial handle as shown in~\cite[Corollary 3, Section 3]{Boyle} or~\cite[Lemma 3]{baykur-sunukjian:surfacestabilization}. 

\end{remark}

\subsection{Annulus rim surgery}\label{S:annulusrimsurgery} Suppose that there is a
smoothly embedded annulus $M(\cong S^1\times I$, where $I$ denotes an interval $[-1,1]$) in $X$ such that
$M$ meets $\Sigma$ normally along $\partial M$ so that $M\cap \Sigma=\partial M$ are two curves $\alpha_{-1}$ and $\alpha_1$ on $\Sigma$. We assume that $\Sigma-\{\alpha_{-1}, \alpha_1\}$ is connected. Choose a trivialization $\nu(M)\to (S^1\times I)\times D^2\cong S^1\times B^3$ such that $M\cong (S^1\times I)\times \{0\}$ and $\nu(M)|_{\Sigma}\cong S^1\times f$, where $f$ denotes a disjoint union of two unknotted segments $\p I\times I\subset I\times D^2=B^3$, a part of the boundary of a trivially embedded band $b=I\times I$ in $B^3$ (See~\figref{finashinsurgery}). So, $M$ is identified with $S^1\times I\times \{0\}$ in $S^1\times b\subset S^1\times B^3$. 

\fig{140}{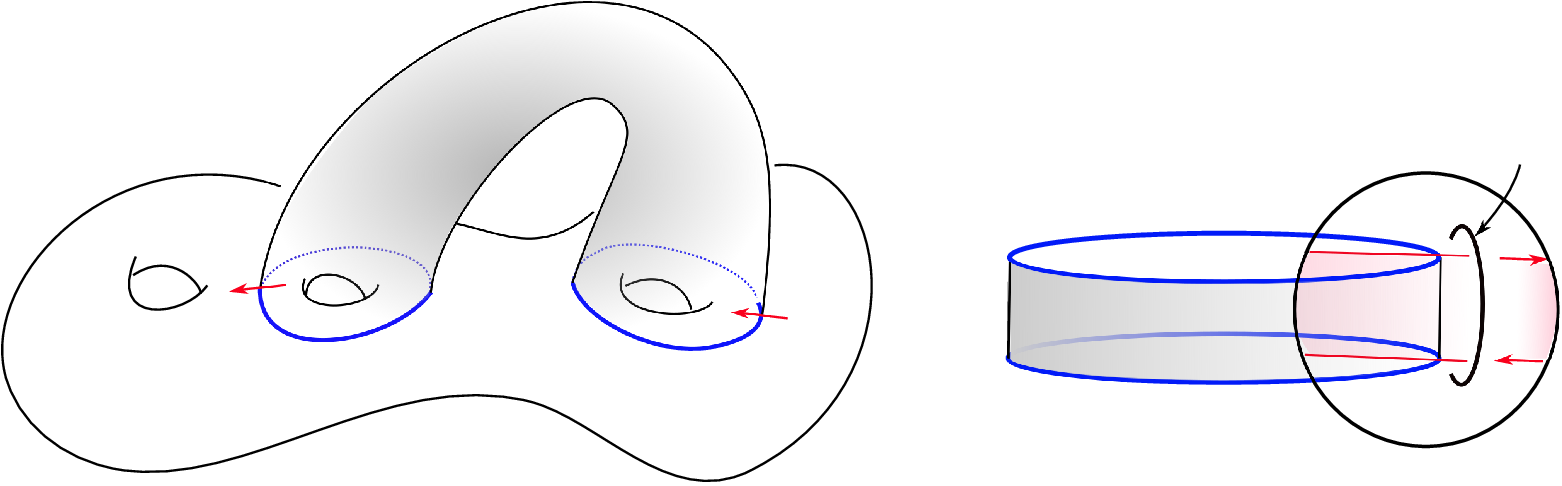}{
  \put(-385,43){\footnotesize$\color{red}I$}
   \put(-350,35){\footnotesize$\color{blue}{\alpha_{1}}$}
  \put(-230,35){\footnotesize$\color{red}I$}
   \put(-255,30){\footnotesize$\color{blue}{\alpha_{-1}}$}
 \put(-150,43){\footnotesize$M$}
  \put(-10,100){\footnotesize$m_b$} 
 \put(-50,0){\footnotesize$(B^3,f)$} 
  \put(-13,50){\footnotesize$\color{red}b$}
  \put(-300,120){\footnotesize$M$}
  \put(-430,20){\footnotesize$\Sigma$}
\caption{$T=S^1\times m_b\subset S^1\times (B^3,f)\cong\nu(M)$ }
\label{F:finashinsurgery}}

Denote by $m_b$ a meridian of $b$ in $B^3$ and let $T$ be a torus in $\nu(M)$ corresponding to $S^1\times m_b\subset S^1\times (B^3, f)$. Knot surgery along this torus $T$ produces a new surface $\Sigma_K(\varphi)$. The simplest gluing $\varphi:\p\nu(T)\to S^1\times \p E(K)$, given by 
 $[S^1]\mapsto [S^1]$, $m_b\mapsto \mu_K$, and $\mu_T\mapsto \lambda_K$, provides the Finashin's annulus rim surgery. This operation obviously yields a band $b_K\subset B^3$ by knotting the band $b$ along $K$ and let $f_K$ be the pair of arcs bounding $b_K$. Here the framing of $b_K$ is chosen the same as the framing of $b$. So the resulting manifold of the annulus rim surgery performed on $\nu(M)\cong S^1\times (B^3,f)$ becomes $S^1\times (B^3,f_K)$ and we write a new pair as follows:
\begin{equation}\label{finashin} 
(X,\Sigma_{K}(\varphi))=(X,\Sigma)-S^1\times (B^3,f)\cup S^1\times (B^3,f_K)
\end{equation}
This construction can be further modified by twists along a meridian and a longitude of $K$, but we will stick to Finashin's construction; see~\cite{kim-danny:topotriviality} for other modifications. Note that when $\pi_1(X-\Sigma)=\Z_d$, any (twisted or otherwise) annulus rim surgery preserves surface knot groups~\cite{finashin},~\cite[Proposition 3.3]{kim-danny:topotriviality}.


\section{Basic Construction}\label{S:basicconstruction}
In order to get our main theorems,  for two knots $K$ and $K'$ related by a single crossing change we will show the $1$-stable equivalence on 
surface knots $\Sigma_{K}(\varphi)$ and $\Sigma_{K'} (\varphi)$ produced by knot surgery.
The complete proof for each knotting construction will be given in~\secref{stabtwistrimsurgeryonecrossinng}, ~\ref{S:stabrimsurgeryonecrossing}, ~\ref{S:stabannulusrimsurgeryonecrossing} but in this section we first present the key constructions and properties that will be used repeatedly in the proofs of our stabilization results.

\fig{80}{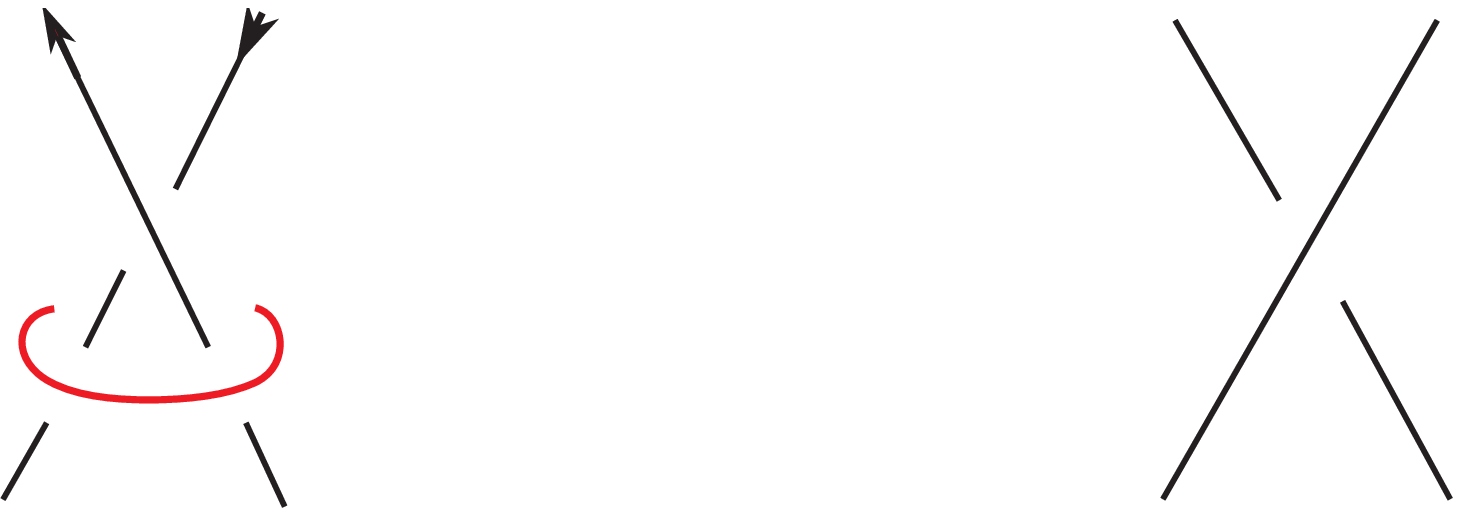}{
  \put(-240,25){\footnotesize$\color{red}c$}
  \put(-180,25){\footnotesize$\color{red}\pm{1}$}
  \put(-60,40){\footnotesize$K'$}
  \put(-270,40){\footnotesize$K$}
\caption{$\pm{1}$-Dehn surgery along $c$ at a crossing}
\label{F:Dehntwist}}

Suppose that two knots $K$, $K'$ in $S^3$ differ by a single crossing change, so that the knot $K'$ is obtained by performing a $\pm{1}$-Dehn surgery along a curve $c$ around an oppositely oriented crossing of $K$ as in~\figref{Dehntwist}. Let $(X_{K}(\varphi), \Sigma_{K}(\varphi))$, $(X_{K'}(\varphi), \Sigma_{K'}(\varphi))$ be two pairs obtained by a knot surgery along a torus $T\subset X-\Sigma$ and gluing map $\phi$ along the knots $K$ and $K'$ respectively. Then we begin by showing that these pairs are related by a torus surgery:

\begin{lemma}\label{L:log}
A log transform of multiplicity $\pm{1}$ performed on the pair $(X_{K}(\varphi),\Sigma_{K}(\varphi))$ produces $(X_{K'}(\varphi),\Sigma_{K'}(\varphi))$.
\end{lemma}
\begin{proof}

We first recall that the knot surgered pair is defined as follows:
$$
(X_{K}(\varphi),\Sigma_{K}(\varphi))=(X,\Sigma)-\nu(T)\cup_{\phi}S^1\times E(K).
$$

Here, we denote $T_{c}$ a torus  $S^1\times c$ in $S^1\times E(K)$ of this decomposition, where $c$ is a curve at a crossing of $K$ as in \figref{Dehntwist}. Identify a neighborhood $\nu(T_c)$ with $S^1\times (c\times D^2)$, where $c\times D^2$ is a neighborhood of $c$ in $E(K)$, and we perform the $\pm{1}$-log transform parallel to the curve $c$ on $T_{c}$ in $(X_{K}(\varphi),\Sigma_{K}(\varphi))$ which is given by the identity in the $S^1$ direction times the $\pm{1}$-Dehn surgery along $c\subset E(K)$. Note that this construction realizes performing a torus surgery along $T_c$ in $S^1\times E(K)$ and gluing back this manifold in $(X,\Sigma)-\nu(T)$ along their boundaries via $\phi$. The resulting manifold is easily identified with $(X_{K'}(\varphi),\Sigma_{K'}(\varphi))$ because there is an obvious diffeomorphism from the log transform of $(S^1\times E(K))$ along $T_c$, denoted by $(S^1\times E(K))_{T_c}$, to $S^1\times E(K')$ which carries each element in a basis $\{[S^1],\mu_{K}, \lambda_{K}\}$ of $H_1(\p (S^1\times E(K))_{T_c})$ to each element in $\{[S^1],\mu_{K'}, \lambda_{K'}\}$ of $H_1(S^1\times \p E(K'))$ respectively.

\end{proof}
\begin{remark}
When the knot surgery $(X,\Sigma)\to (X,\Sigma_{K}(\varphi))$ is an ambient surgery such as (twisted) rim surgery and annulus rim surgery, the above torus surgery on $(X,\Sigma_{K}(\varphi))$ gives rise a new embedding $\Sigma_{K'}(\varphi)$ in $X$. And, because $\pi_1(X_{K'}(\varphi)-\Sigma_{K'}(\varphi))\cong \pi_1(X_{K}(\varphi)-\Sigma_{K}(\varphi))$ observed from the proof in~\lemref{log}, if a knot surgery $(X,\Sigma)\to (X,\Sigma_{K}(\varphi))$ does not change the surface knot group $\pi_1(X-\Sigma)$ under some suitable circumstances then the new embedding $\Sigma_{K'}(\varphi)$ also preserves its knot group.
\end{remark}

\subsection{\bf Fiber sum and gluing map}\label{S:fibersum}
A torus surgery of $(X_{K}(\varphi),\Sigma_{K}(\varphi))$ along $T_c$ can be described as a {\em fiber sum} of $(X_{K}(\varphi),\Sigma_{K}(\varphi))$ and $S^1\times S^3$: Let $T_u$ be a standardly embedded torus $S^1\times u$ in $S^1\times S^3$, where $u$ is an unknot in $S^3$. Then we write the torus surgered manifold in~\lemref{log} as a fiber sum along tori $T_c$ and $T_u$: 
\begin{equation}\label{fibersum}
(X_{K'}(\varphi),\Sigma_{K'}(\varphi))\cong (X_{K}(\varphi),\Sigma_{K}(\varphi))\fcs{T_{c}}{T_u}S^1\times S^3.
\end{equation}

In order to describe the gluing map $f:\p\nu(T_{u})\to \p \nu(T_{c})$ carefully, identify a tubular neighborhood $\nu(T_{u})$ in $S^1\times S^3$ with $S^1\times (u\times D^2)$, where $u\times D^2$ is a neighborhood $\nu(u)$ of $u$ in $S^3$. Let $a=S^1\times \{\text{pt}\}\subset S^1\times u=T_u$ and $a'$ denotes its pushoff into $\p\nu(T_u)$. Then $\{a', m_u, l_u\}$ forms a basis for $H_1(\p\nu(T_{u}))$ where $m_u$, $l_u$ are a meridian-longitude pair of $u$ with respect to the identification of $\nu(u)$. Similarly, under an identification $\nu(T_{c})\cong S^1\times (c\times D^2)$ in $S^1\times E(K)$, let $\gamma=S^1\times \{\text{pt}\}\subset S^1\times c=T_c$ and $m_{c}$, $l_{c}$ be a meridian-longitude pair of $c$. So $\{\gamma', m_{c}, l_{c}\}$ gives a basis for $H_1(\p\nu(T_{c}))$, where $\gamma'$ is a pushoff of $\gamma$ into $\p\nu(T_{c})$.

Described in the proof of \lemref{log}, this construction realizes the product of a $\pm 1$-Dehn surgery with $S^1$, and note that the meridian $m_u$ (longitude $l_u$) of $u$ is the longitude (meridian) of the solid torus $S^3-\nu(u)$ that is glued into $E(K)-\nu(c)$. So the gluing map $f: S^1\times (u\times \p D^2)\to S^1\times (c\times \p D^2)$ is determined as follows;
\begin{equation}\label{f:gluing}
f_*(a')=\gamma',  \quad     f_*(m_u)=m_{c} , \quad \text{and}  \quad   f_*(l_u)=\pm  m_{c}+l_{c}.
\end{equation}

In our present purpose, it is important to keep track of the gluing map in this fiber sum, from which we can determine the framing arising in our proof of the stabilization result.
\subsection{\bf Cobordism}\label{S:cobordism}
In this section, we will construct a cobordism $W$ whose upper boundary is $(X_{K}(\varphi),\Sigma_{K}(\varphi))\fcs{T_{c}}{T_u}S^1\times S^3$ from $(X_{K}(\varphi),\Sigma_{K}(\varphi))\sqcup S^1\times S^3$. The proof of stable equivalence for surfaces $\Sigma_{K}(\varphi)$ and $\Sigma_{K'}(\varphi)$ will come from the middle level of the constructed cobordism $W$. 

Given $(X_{K}(\varphi),\Sigma_{K}(\varphi))$ and $S^1\times S^3$, containing tori $T_{c}$ and $T_{u}$  respectively, we obtain $W$ by forming $((X_{K}(\varphi),\Sigma_{K}(\varphi))\sqcup S^1\times S^3)\times I$, and attaching a `doubly round $1$-handle' $[-1, 1]\times T^2\times D^2$ to the upper boundary. In particular, $\{-1\}\times T^2\times D^2$ is attached to $\nu(T_u)$ while $\{1\}\times T^2\times D^2$ is glued to $\nu(T_c)$; the attaching map in the first case is the `identity' with respect to some identification of $\nu(T_u)$, while the second torus is attached by a diffeomorphism that should realize the gluing map $f$ described in~\eqref{f:gluing} when restricted to the boundary. From this description, it follows that in the upper boundary of W the meridian $m_u$ of $T_u $ is identified with the meridian $m_c$ of $T_c$ because the attaching map $\{1\} \times T^2\times D^2\to\nu(T_c)$ necessarily preserves the normal disk.

\fig{140}{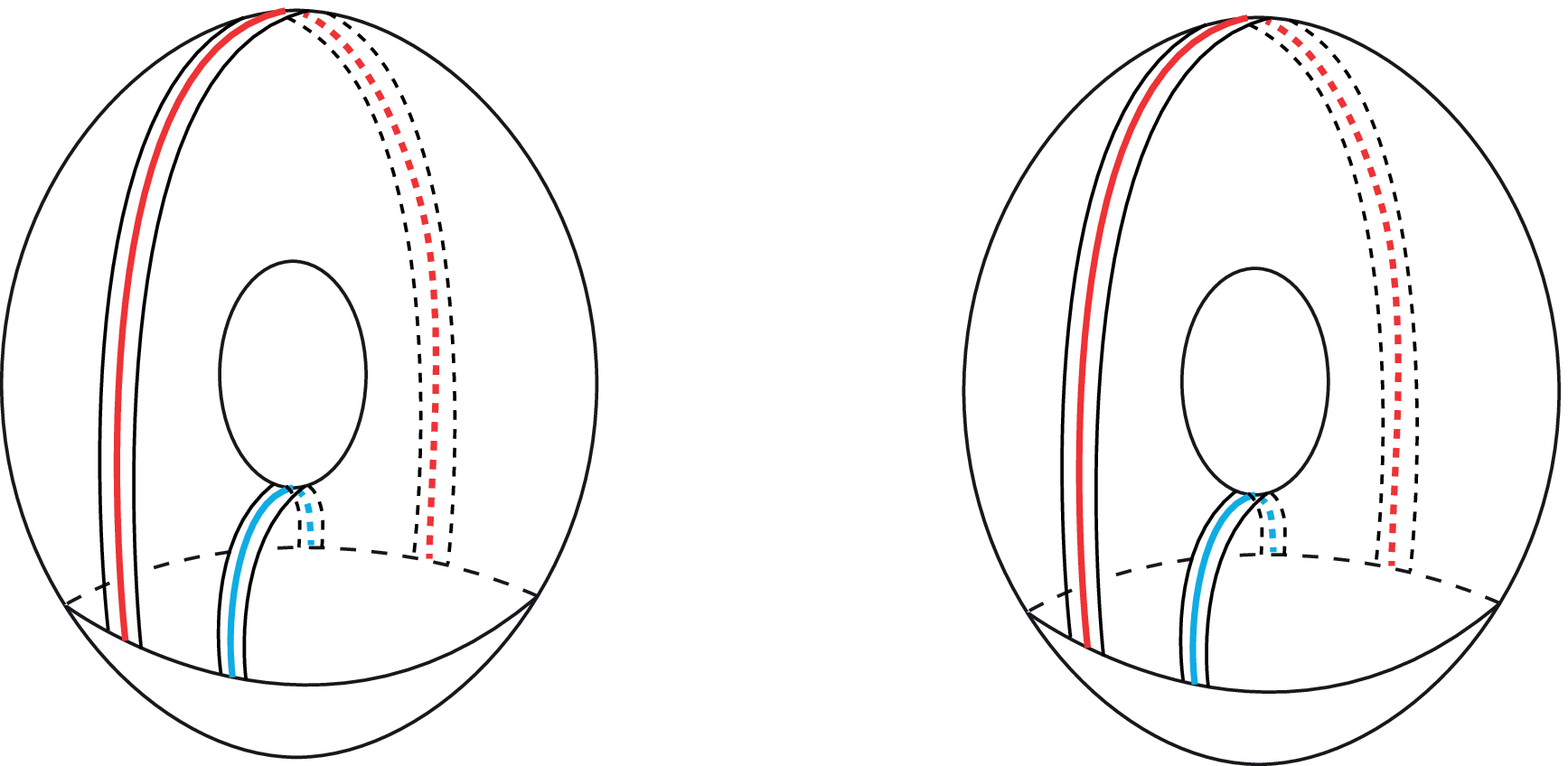}{
  \put(-240,4){\footnotesize$h^0$}
  \put(-238,25){\footnotesize$h_{\gamma}^1$}
  \put(-283,60){\footnotesize$h_{c}^1$}
   \put(-240,110){\footnotesize$h^2$}
  \put(-55,4){\footnotesize$h^0$}
    \put(-105,60){\footnotesize$h_{u}^1$}
  \put(-60,25){\footnotesize$h_{a}^1$}
   \put(-60,110){\footnotesize$h^2$}
\caption{Handle decompositions of $T_c$ and $T_u$}
\label{F:Torihandles}}

We also explicitly give a handle by handle description of $W$ for later work according to a standard handle structure of tori $T_{c}$ and $T_u$ and the gluing between them. 
As in~\figref{Torihandles} (thicken by $D^2$), a standard handle decomposition of $\nu(T_{c})=T_{c}\times D^2$ can be given by one $0$-handle $h^0$, two $1$-handles $h_{\gamma}^1$, $h_{c}^1$, and one $2$-handle $h^2$, where $h_{\gamma}^1$ and  $h_{c}^1$ denote the $1$-handles induced by the first and second factors of $T_c=S^1\times c$ respectively. 
Similarly, $\nu(T_u)=T_u\times D^2$ has one $0$-handle $h^0$, two $1$-handles $h_{a}^1$, $h_{u}^1$ generated by the first and second factors of $T_u=S^1\times u$, and one $2$-handle $h^2$. From the handles of these tori, $W$ will be built by adding one $5$-dimensional $1$-handle $H^1$, two $2$-handles denoted by $H_{\gamma a}^2$, $H_{c u}^2$, and one $3$-handle $H^3$ to $((X_{K}(\varphi),\Sigma_{K}(\varphi))\sqcup S^1\times S^3)\times I$.

To examine this attaching process closely which is basically same as the previous description, note that the handle structure of a neighborhood of a torus simply comes from each $2$-dimensional $k$-handle $h^k$ of the torus so that the corresponding $4$-dimensional $k$-handle is of the form $h^k\times D^2=(D^k\times D^{2-k})\times D^2$.  Then we define a $5$-dimensional $(k+1)$-handle $H^{k+1}$ by $(I\times D^k)\times D^{2-k}\times D^2$ where $I$ denotes the interval $[-1,1]$, and the attaching process of a $(k+1)$-handle $H^{k+1}$ is described as follows. The discs $(\{-1\}\times D^k)\times 0$ and $(\{1\}\times D^k)\times 0\subset \p (I\times D^k)\times 0$ are attached to each core of the $k$-handles of $T_u$ and $T_c$ respectively, and the rest $(I\times \p D^k)\times 0$ connects the boundaries of these cores in $\p_{+} W_{k}$, where $W_k$ denotes a handlebody obtained by attaching all handles of index $\le k$. Moreover, as described before, $\{-1\}\times D^k\times D^{2-k}\times 0$ and $\{1\}\times D^k\times D^{2-k}\times 0$ are glued to each $2$-dimensional $k$-handle $h^k$ of $T_u$ and $T_c$ respectively. The boundary of the normal bundle of the torus restricted over the handle $\{-1\}\times D^k\times D^{2-k}\times \p D^2$ is glued to $\p\nu(T_u)$ by the identity, and $\{1\}\times D^k\times D^{2-k}\times \p D^2$ is glued to $\p\nu(T_c)$ by the diffeomorphism $f$ so that it gives rise to the framing.

\vskip6pt
 
We will find out the resulting upper boundary at each stage of attaching handles while we're building a cobordism $W$ from $(X_{K}(\phi),\Sigma_{K}(\phi))\sqcup S^1\times S^3$ to the fiber sum $(X_{K}(\phi),\Sigma_{K}(\phi))\fcs{T_{c}}{T_u}S^1\times S^3$. The level of $W$ after a $1$-handle $H^1$ is obviously the connected sum $(X_{K}(\phi),\Sigma_{K}(\phi))\cs S^1\times S^3$, and for the rest handles we will give more careful arguments. 

\vskip6pt
Since we're interested in the equivalence of embeddings $\Sigma_{K}(\varphi)$ and $\Sigma_{K'}(\varphi)$ in a same manifold $X$, one may focus on ambient surgery so that $X_{K}(\varphi)\cong X$. Then all constructions of (twisted) rim surgery and annulus rim surgery will share the following diagram which indicates the boundary at each stage of adding handles:
\begin{equation}\label{diagram}
\begin{split}
  (X,  &\Sigma_{K}(\phi))\sqcup  S^1\times S^3\xrightarrow{H^1}
  (X\cs S^1\times S^3,\Sigma_{K}(\phi))\xrightarrow{H_{\gamma a}^2}
  (X\cs S^4,\Sigma_{K}(\phi))\xrightarrow{H_{cu}^2}\\
  &(X\cs\sst,\Sigma_{K}(\phi))\xrightarrow{H^3}
  (X,\Sigma_{K}(\phi))\fcs{T_{c}}{T_u}S^1\times S^3\cong (X,\Sigma_{K'}(\phi)).
\end{split}  
\end{equation}

Our main argument for $1$-stable equivalence of $\Sigma_{K}(\varphi)$ and $\Sigma_{K'}(\varphi)$ follows from the middle level of $W$:
Let $W_k$ be a handlebody obtained by attaching all handles of index $\le k$ in $W$. Then from the diagram~\eqref{diagram}, $\p _{+}W_2$ is $(X\cs\sst,\Sigma_{K}(\phi))$, which will be shown in Lemma~\ref{L:S2bundleS2}, Theorem~\ref{T:onecrosssdiffeorim}, and~\ref{T:onecrosssdiffeoannuls}. After adding a $3$-handle to $W_2$, we would have a cobordism $W$ from $(X,\Sigma_{K}(\phi))\sqcup S^1\times S^3$ to the fiber sum which is diffeomorphic to $X$ itself containing the surface $\Sigma_{K'}(\phi)$ by~\lemref{log} and~\eqref{fibersum}. Turning the $3$-handle $H^3$ upside down so that it becomes to attach a $2$-handle $H^2_*$ to the fiber sum gives a connected sum with $\sss$ or $\sst$ on $X$. Disregarding the surface $\Sigma_{K'}(\phi)$ in the fiber sum, we first note that the level after attaching the $2$-handle $H^2_*$ to $X$ is same as $\p_{+}W_2$ which is diffeomorphic to $X\cs\sst$. But since we're building a relative cobordism, it has to be argued that attaching the $2$-handle $H^2_*$ gives rise to {\em the pair} $(X\cs\sst,\Sigma_{K'}(\varphi))$ on the boundary. This will be verified in Lemma~\ref{L:dual3handle}, Theorem~\ref{T:onecrosssdiffeorim}, and~\ref{T:onecrosssdiffeoannuls} so that it will prove the $1$-stable equivalence of $\Sigma_{K}(\varphi)$ and $\Sigma_{K'}(\varphi)$.

\vskip6pt
When we discuss about the stabilization in~\secref{stabcylicknotsurgery} for the case that knot surgery $(X,\Sigma)\to (X_{K}(\varphi),\Sigma_{K}(\varphi))$ is cyclic, that is a surgery preserving $\pi_1(X_{K}(\varphi)-\Sigma_{K}(\varphi))\cong\pi_1(X-\Sigma)$ as a cyclic group, the cobordism $W$ will be considered to be from $(X_{K}(\varphi),\Sigma_{K}(\varphi))\sqcup S^1\times S^3$ and it will be shown that the diagram~\eqref{diagram} also works for this. 

\vskip6pt
In the following subsections, we will investigate the level of $W$ at each step of adding handles, and give some assertions that will be used in the proof of the stabilization for each knotting construction. For the purpose in this article, the knot surgery $(X,\Sigma)\to (X_{K}(\varphi),\Sigma_{K}(\varphi))$ is assumed to be cyclic or be an ambient surgery to produce a new surface $\Sigma_{K}(\varphi)$ in $X$ throughout the rest of paper, although some proofs may work for more general cases.   

\vskip10pt

\subsubsection{\bf Attaching a $2$-handle $H_{\gamma a}^2$ }\label{2handlesec1}

Note that the homotopy class $\gamma a$ of attaching circle of $2$-handle $H_{\gamma a}^2$ is represented by a curve $\gamma+a$ in the outer boundary of $((X_{K}(\varphi),\Sigma_{K}(\varphi))\sqcup S^1\times S^3)\times I\cup H^1$ as depicted in~\figref{onehandleisotopy1}. We claim that the resulting manifold on the boundary is $(X_K(\varphi)\cs S^4,\Sigma_K(\varphi))$. 
\fig{120}{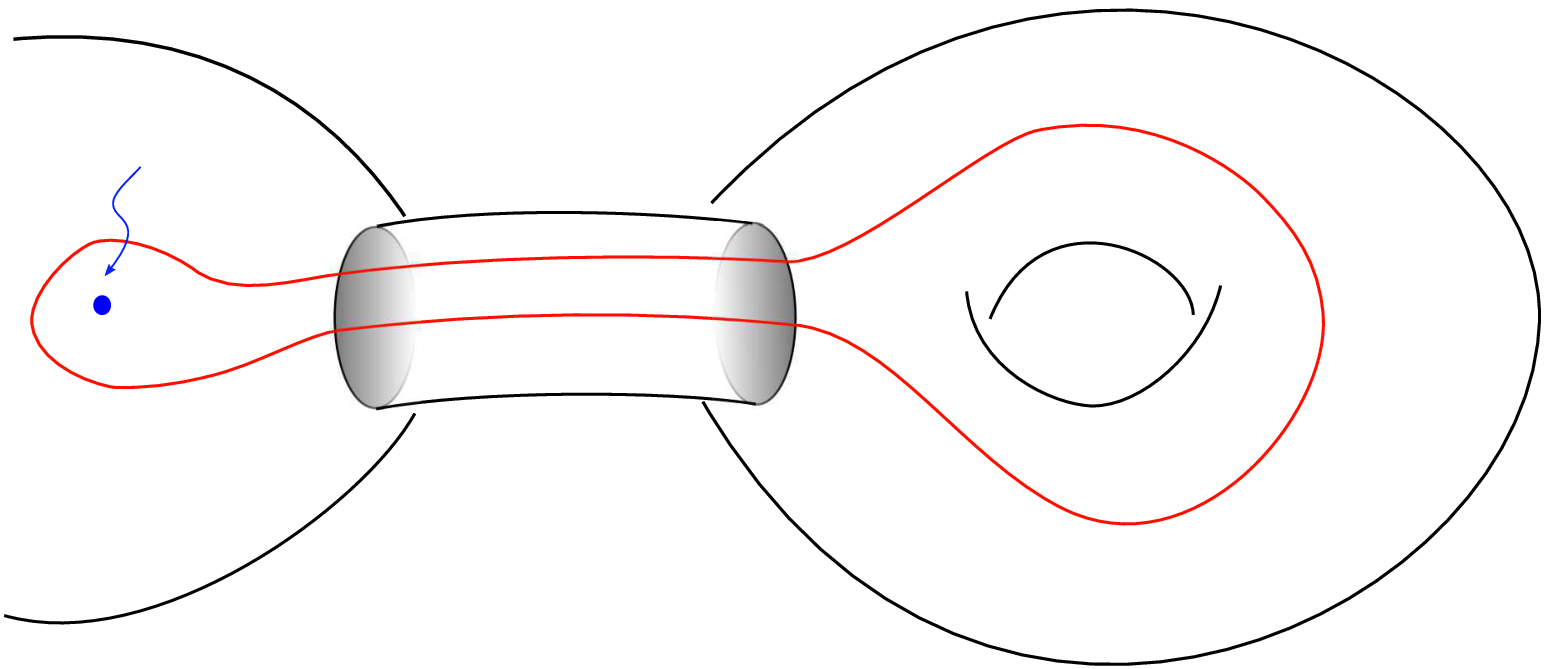}{
  \put(-350,30){\footnotesize$(X_K(\varphi),\Sigma_K(\varphi))\times\{1\}$}
    \put(-280,95){\footnotesize$\color{blue}{\Sigma_K(\varphi)}$}
   \put(-180,90){\footnotesize$H^1$}
 \put(-112,20){\footnotesize$\color{red}\gamma+a$}
   \put(-10,110){\footnotesize$S_a^1\times S^3\times\{1\}$}
\caption{Attached $1$-handle $H^1$ to $((X_K(\varphi),\Sigma_K(\varphi))\sqcup S^1\times S^3)\times\{1\}$}
\label{F:onehandleisotopy1}}

\begin{lemma}\label{L:2handle1} 
The resulting upper boundary of attaching a $2$-handle $H_{\gamma a}^2$ to $W_1$ is diffeomorphic to $(X_K(\varphi)\cs S^4,\Sigma_K(\varphi))$. 
\end{lemma}
\begin{proof}
We first draw $((X_K(\varphi),\Sigma_{K}(\varphi))\sqcup S^1\times S^3)\times I\cup H^1$ as \figref{onehandleisotopy3}, which is basically obtained by attaching a `round handle' $R:=S_+^1\times S^3\times I$, where $S_{+}^1$ denotes a $1$-handle of $S^1$, to the outer boundary $(X_K(\varphi),\Sigma_{K}(\varphi))\times\{1\}$ along $\partial{S_+^1}\times S^3\times I$.
One can see that $\gamma+a$ is isotoped in $X_{K}(\varphi)\cs S^1\times S^3-\Sigma_{K}(\varphi)$ to $a$ as demonstrated in~\figref{onehandleisotopy4} so that it yields a pairwise diffeomorphism $(X_{K}(\varphi)\cs S^1\times S^3,\Sigma_{K}(\varphi), \gamma+a)\to  (X_{K}(\varphi)\cs S^1\times S^3,\Sigma_{K}(\varphi), a)$. 
Since attaching a $2$-handle along $\gamma+a$ gives the effect on the boundary that surgers out the curve $a=S^1\times \text{pt}$ of the second summand $S^1\times S^3$ in $X_K(\varphi)\cs S^1\times S^3$, so the result follows. 

\fig{150}{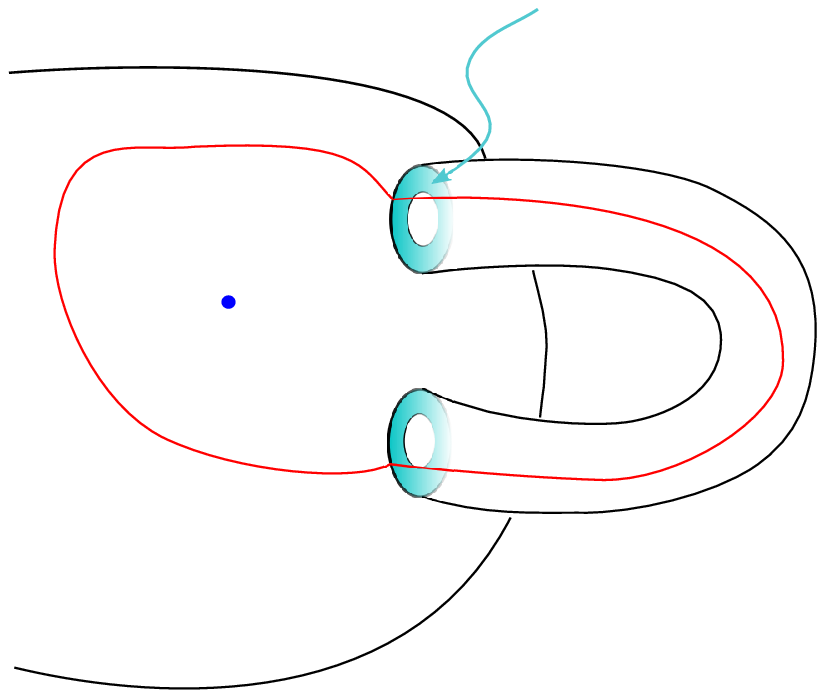}{
\put(-180,20){\footnotesize$(X_K(\varphi),\Sigma_K(\varphi))\times\{1\}$}
\put(-430,50){\footnotesize$\color{red}\gamma+a$}
\put(-130,70){\footnotesize$\color{blue}{\Sigma_K(\varphi)}$}
\put(-50,150){\footnotesize$\text{pt}\times S^3\times I\subset \partial{S_+^1}\times S^3\times I$}   
\put(-50,25){\footnotesize$R$}   
\put(-140,110){\footnotesize$\color{red}\gamma+a$}
\caption{$ ((X_K(\varphi), \Sigma_{K}(\varphi))\sqcup S^1\times S^3)\times I\cup H^1$}
\label{F:onehandleisotopy3}}

\fig{130}{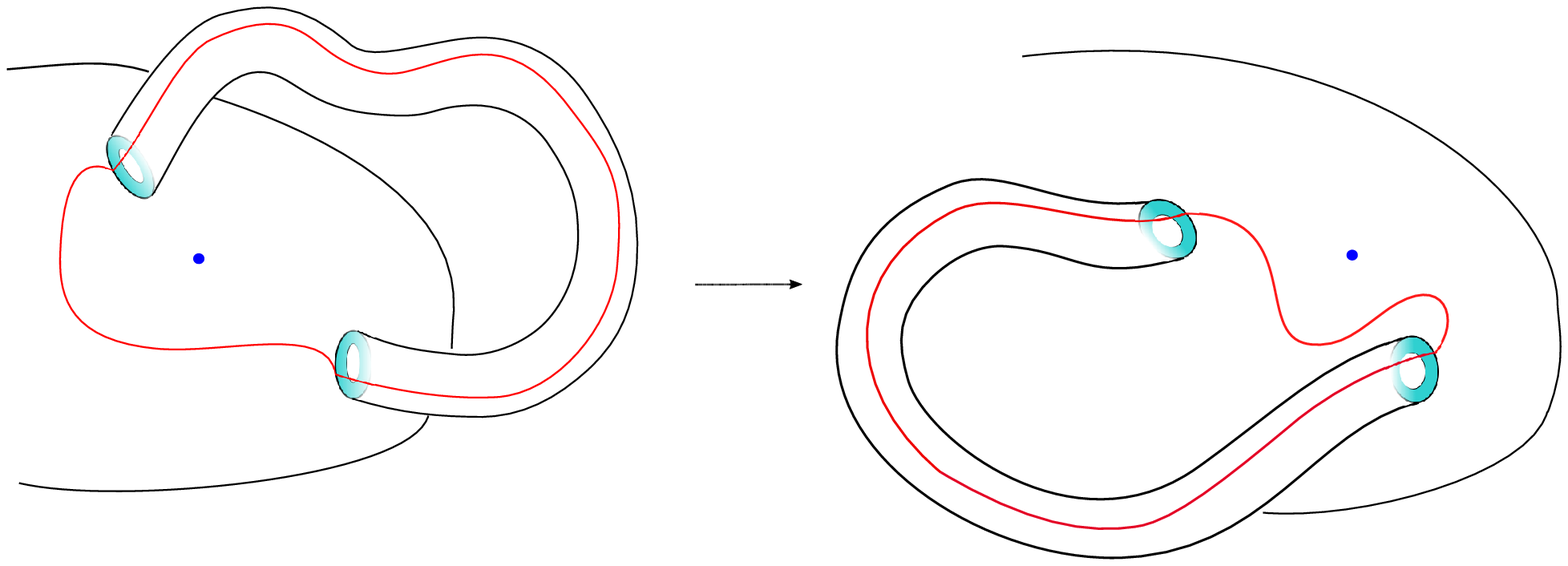}{
\put(-320,75){\footnotesize$\color{blue}{\Sigma_K(\varphi)}$}
\put(-130,100){\footnotesize$X_K(\varphi)$}
\put(-380,100){\footnotesize$X_K(\varphi)$}   
\put(-370,40){\footnotesize$\color{red}\gamma+a$}
\put(-100,50){\footnotesize$\color{red}\gamma+a$}
\put(-90,-10){\footnotesize$R$}
\put(-270,10){\footnotesize$R$}
\put(-60,80){\footnotesize$\color{blue}{\Sigma_K(\varphi)}$}
\caption{Isotopy of round handle $R$}
\label{F:onehandleisotopy4}}
\end{proof}

\subsubsection{\bf Attaching a $2$-handle $H_{cu}^2$ and a dual handle $H_{*}^2$ of $3$-handle $H^3$}\label{S:23handlesec}
We now deal with the next $2$-handle $H_{cu}^2$ and the dual $2$-handle $H_{*}^2$ of $3$-handle $H^3$ in building $W$. The level of our {\em relative} cobordism $W$ after adding those handles will have more subtle issues on the framing that will depend on each knotting construction, and so the details of the analysis for the boundary will be referred to the next following sections. But here we will first focus on the ambient manifold to study the boundary after adding the handles without concerning surfaces. 

As seen in Lemma~\ref{L:2handle1}, the level of $W$ after adding a $2$-handle $H_{\gamma a}^2$ is diffeomorphic to $X_K(\varphi)\cs S^4\cong X_K(\varphi)$. Note that any knot surgery $(X,\Sigma)\to (X_{K}(\varphi),\Sigma_{K}(\varphi))$ preserves the fundamental group of the ambient manifold, so $X_{K}(\varphi)$ is simply-connected and adding another $2$-handle $H_{cu}^2$ along the curve $c+u$ in $X_K(\varphi)\cs S^4$ gives rise a connected sum with a $S^2$-bundle over $S^2$ on $X_K(\varphi)$. The following lemma determines the framing.

\begin{lemma}\label{L:2handlesurgeryframing}
Attaching a $2$-handle $H_{cu}^2$ to $X_K(\varphi)$ provides a connected sum with the twisted $S^2$-bundle over $S^2$ so that $\p_+ W_2$ is diffeomorphic to $X_K(\varphi)\cs \sst$.  
\end{lemma}
\begin{proof}
Since $c$ is nullhomoptic in $X_{K}(\varphi)$, it bounds a disk $D_0$ in $X_K(\varphi)$ that may be assumed to be embedded in closed $4$-manifolds and intersect with the surface $\Sigma_K(\varphi)$. As described in~\secref{cobordism} about $W$, the 2-handle $H_{cu}^2=D^2\times D^3 = (I\times D^1)\times D^3$ is attached so that $\{\pm 1\}\times D^1\times 0$ are glued to the cores of the 1-handles $h_u^1$, $h_{c}^1$ in the tori $T_u$, $T_c$. And, we split the $D^3$ as $I\times D^2$ so that $\{\pm 1\}\times D^1\times I\times 0$ is glued to the $2$-dimensional $1$-handle $h^1$ of each torus $T_u$, $T_c$ and the normal $\{\pm 1\}\times D^1\times I \times D^2$ corresponds to the normal bundle of the torus restricted over the handle. So the gluing map $f_*(l_u)=\pm  m_{c}+l_{c}$ in~\eqref{f:gluing}, expressed with the meridian-longitude defined from the disk $D_0$, gives rise to the framing of attaching this $2$-handle which is same as the framing of the surgery along the curve $c$ on the boundary. It verifies that the framing relative to the disk $D_0$ is odd, and therefore the surgery gives the twisted $S^2$-bundle over $S^2$.

\end{proof}

Turning the cobordism $W$ upside down, denoted by $W^*$, yields a dual $2$-handle $H_{*}^2$ of the $3$-handle which is attached to a collar $\p_{+}\overline W\times I=\p_{-} W^*\times I\cong X_{K'}(\varphi)\times I$. Again this attaching will give a $S^2$-bundle over $S^2$ on $X_{K'}(\varphi)$ as follows.

\begin{lemma}\label{L:dualhandleattachingcircle}
The level of $W^*$ after adding a dual $2$-handle $H_{*}^2$ is diffeomorphic to $X_{K'}(\varphi)\cs \sst$. 
\end{lemma}

\begin{proof}
The $3$-handle $H^3=D^3\times D^2=I\times D^2\times D^2$ is attached in the same way of the proof of~\lemref{2handlesurgeryframing} as the parts $\{\pm 1\}\times D^2\times 0$ of attaching sphere are glued to the $2$-handles $h^2_{T_u}$, $h^2_{T_c}$ of the tori $T_u$, $T_c$ respectively. And the normal disk of $T_u$ is a cocore of $H^3$ whose boundary will be the attaching circle of its dual $2$-handle $H_{*}^2$ glued in $\p\nu(T_c)$ according to $f$. So as given in~\eqref{f:gluing} $f_*(m_u)=m_c$, the $2$-handle $H_{*}^2$ is attached to a meridian circle to $c$ in $S^1\times E(K)\subset X_{K}(\varphi)$. Since the fiber sum of $X_{K}(\varphi)$ with $S^1\times S^3$ is diffeomorphic to $X_{K'}(\varphi)$, we need to see where the dual $2$-handle is attached in $X_{K'}(\varphi)$. It follows from that the diffeomorphism is done by $\pm 1$-Dehn twist along the curve $c$ in $E(K)$ trivially multiplied by $S^1$. Placing a meridian circle $m_c$ to $c$ in~\figref{Dehntwist}, then blowing down to arrive at the right hand side of that figure: in the process, the meridian circle becomes a circle $c'$ that links the crossing of $K'$ in the same way that $c$ links the crossing of $K$. Moreover, this picture verifies the fact that the framing on $c'$ should be odd, since the dual $2$-handle (corresponding to the meridian of $c$) will have framing $0$, which becomes $\pm1$ after blowing down. 
\end{proof}

In the next following sections, we will verify the rest process in the digram~\eqref{diagram} and prove the $1$-stable equivalence of the surfaces $\Sigma_{K}(\varphi)$, $\Sigma_{K'}(\varphi)$ according to each knotting construction of surfaces.
\section{$1$-stable equivalence of knotted surfaces}\label{S:proofThmABC}

\subsection{\bf Twist rim surgery}\label{S:stabtwistrimsurgeryonecrossinng}
Let $\pi_1(X-\Sigma)$ be any group $G$, and suppose that the surface $\Sigma\subset X$ carries a nontrivial homology class with $H_1(X-\Sigma)=\Z_d$. Then our key theorem is stated as follows:

\begin{theorem}\label{T:onecrosssdiffeotwist}
Suppose that two knots $K$, $K'$ in $S^3$ differ by a single crossing change. If $(X,\Sigma_{K}(m))$ is a pair produced by an $m$-twist rim surgery such that either $m=\pm 1$ or the meridian $\mu_{\Sigma}$ has order $d$ in $\pi_1(X-\Sigma)$ and $(m,d)=1$, then $(X\cs\sst,\Sigma_{K}(m))$ is pairwise diffeomorphic to $(X\cs\sst, \Sigma_{K'}(m))$.
\end{theorem}

To prove Theorem~\ref{T:onecrosssdiffeotwist}, we now turn to the $2$-handle $H_{cu}^2$ and $3$-handle in the diagram~\eqref{diagram}.


\subsection*{\bf I. Attaching a $2$-handle $H_{cu}^2$ }\label{2handlesec2}

Lemma~\ref{L:2handle1} has shown that the upper boundary after adding a $2$-handle $H_{\gamma a}^2$ is diffeomorphic to $(X\cs S^4,\Sigma_K(m))$, and from~\lemref{2handlesurgeryframing} adding another $2$-handle $H_{cu}^2$ along the curve $c+u$ in $(X\cs S^4,\Sigma_K(m))$ gives $X\cs \sst$ for the ambient manifold. But since our stabilization is performed in the `outside' of the surface $\Sigma_K(m)$, we first show that $c$ is nullhomotopic in $X-\Sigma_{K}(m)$:

\begin{proposition}\label{P:curveC} 
Suppose that $\Sigma\subset X$ is a surface carrying $H_1(X-\Sigma)=\Z/d$. If either $m=\pm 1$ or the meridian $\mu_{\Sigma}$ has order $d$ in $\pi_1(X-\Sigma)$ and $(m,d)=1$, then the curve $c$ is nullhomotopic in $X-\Sigma_{K}(m)$. 
\end{proposition}
\begin{proof}
Since $c$ is a curve at an oppositely oriented crossing of a knot $K$ as in~\figref{Dehntwist}, its homotopy class $c$ can be expressed as $g^{-1}\mu_K^{-1}g\mu_K$ in terms of some element $g\in \pi_1(E(K))$. It easily follows from the presentation~\eqref{mtwistfundgpeq2} for $\pi_1(X - \Sigma_{K}(m))$ in Proposition~\ref{P:twistrimfundagp} that $c=g^{-1}\mu_K^{-1}g\mu_K=1$ when either $m=\pm 1$ or $\mu_{\Sigma}^d=1$ and $(m,d)=1$.

\end{proof}

\begin{remark}\label{R:embedimmerseddisk}
Proposition~\ref{P:curveC} asserts that we can find a disk $D\subset X-\Sigma_{K}(m)$ such that $\p D=c$, and in general it will be an immersed disk. But since the circle $c$ lies the interior of the $4$-manifold $X-\Sigma_{K}(m)$, we simply pipe the self-intersection of $\text{int} (D)$ off of its boundary to make it an embedded disk. It follows that there is an induced diffeomorphism from any surgered manifold of $X$ along $c$ to the connected sum of $X$ with a $S^2$-bundle over $S^2$ i.e. $(X,\Sigma_{K}(m))\cs (\sss,\emptyset)$ or $(\sst,\emptyset)$. 
Note that $(X,\Sigma_{K}(m))\cs (\sss,\emptyset)$ is obtained from $X$ by surgery on $c$ with the framing determined by the unique normal framing of $D$, and $(X,\Sigma_{K}(m))\cs (\sst,\emptyset)$ is obtained with the other framing. To address the framing, we will explicitly find a disk and examine the framing of surgery relative to the disk. We will repeat this argument for the framing in~\lemref{S2bundleS2},~\ref{L:dual3handle} and other constructions~\thmref{onecrosssdiffeorim}, ~\ref{T:onecrosssdiffeoannuls}, and~\thmref{mainthmE}. Note that throughout this paper it may be assumed that the existing disk $D$ bounding the circle $c$ is an embedded disk by using the piping operation in the interior of the complement of surface $\Sigma_{K}(\phi)$. In fact, for our purpose this step is unnecessary in~\lemref{S2bundleS2},~\ref{L:dual3handle},~\thmref{onecrosssdiffeorim}, ~\ref{T:onecrosssdiffeoannuls} as we will see that the existence of immersed disk is sufficient, but it is not harmful to do it.

\end{remark}

\fig{130}{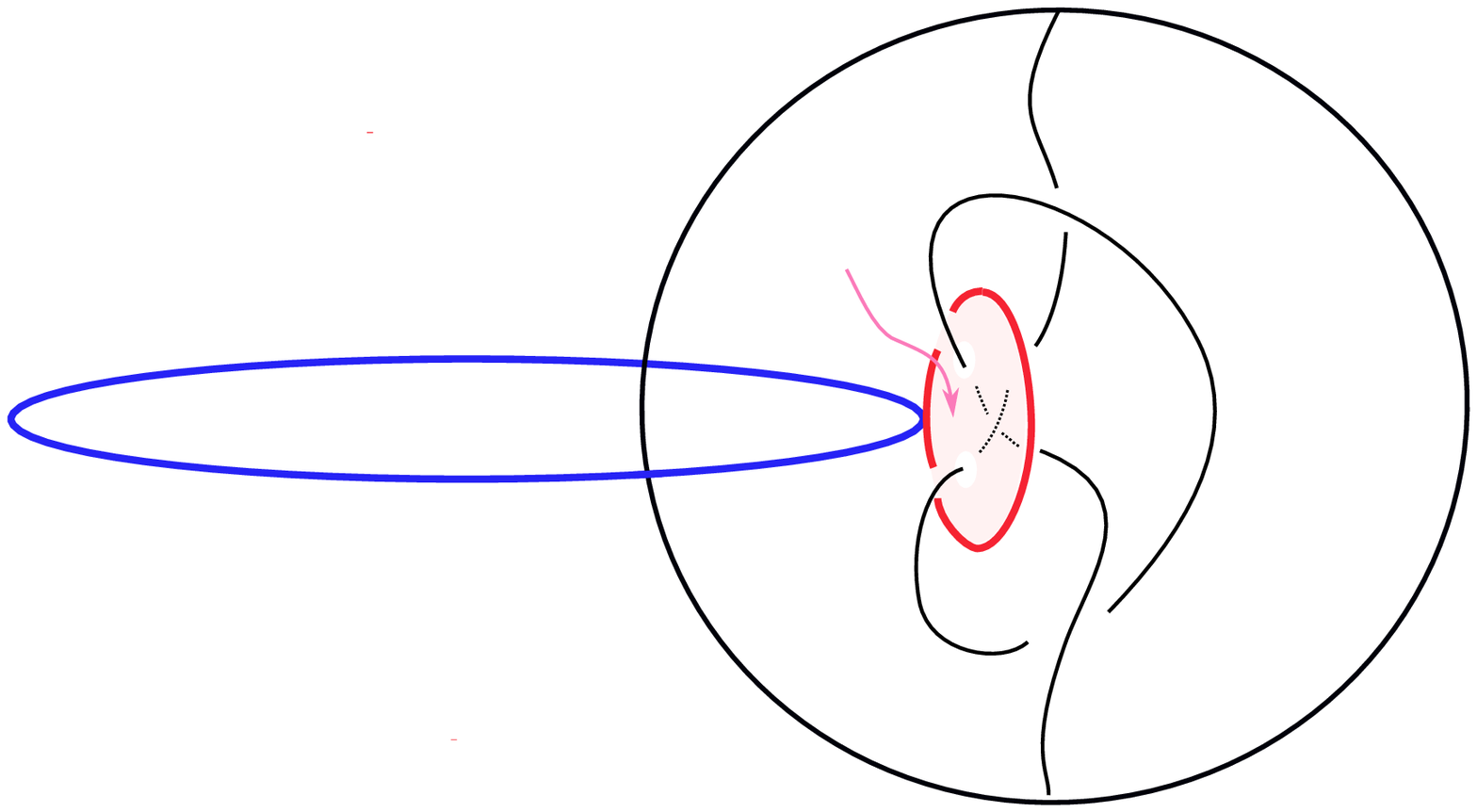}{
\put(-250,70){\footnotesize$\color{blue}{\gamma}$}
\put(-110,90){\footnotesize$D_0$}
\put(-65,65){\footnotesize$\color{red}c$}
\caption{Torus $T_c=\gamma\times c\subset S^1\times_{\tau^m}(B^3,K_+)$}
\label{F:torusc}}

We first recall that twist rim surgery is performed along a rim torus in $\nu(\alpha)\cong S^1\times (B^3,I)$ so that it produces the mapping torus $S^1\times_{\tau^m} (B^3, K_+)$; see~\eqref{twistrimsurgery2}. So we view the torus $T_c=\gamma\times c=S^1\times c\subset S^1\times E(K)$ used in a log transform in~\lemref{log} is lying in the mapping torus $S^1\times_{\tau^m} (B^3, K_+)$; see~\figref{torusc}, and so the curve $c$ is in $S^1\times_{\tau^m} (B^3-K_+)$.

\begin{lemma}\label{L:S2bundleS2} 
The upper boundary of $2$-handlebody $W_2$ in the cobordism $W$ is diffeomorphic to $(X\cs\sst,\Sigma_{K}(m))$. 
\end{lemma}
\begin{proof}
Attaching a $2$-handle $H_{cu}^2$ gives a surgery along $c$ in $(X,\Sigma_{K}(m))\cong (X\cs S^4,\Sigma_K(m))$. This curve $c$ obviously bounds an embedded disk $D_0$ in the component $S^1\times_{\tau^m}(B^3,K_+)$ of the decomposition~\eqref{twistrimsurgery2} for $(X,\Sigma_{K}(m))$ that intersects with $\Sigma_{K}(m)$ at two points as in~\figref{torusc}. Using the disk $D_0$, we denote by $\psi_0$ the surgery framing induced from attaching the $2$-handle $H_{cu}^2$. Note that there is a framing determined by the unique normal framing of $D_0$, but by~\lemref{2handlesurgeryframing}, our framing $\psi_0$ on $c=\p D_0$ relative to this disk is the other one so that it doesn't extend over $D_0$. To investigate the framing on $c$ in $X-\Sigma_{K}(m)$, we shall find a disk $D$ bounding $c$ in the complement of $\Sigma_{K}(m)$, and compare $\psi_0$ with the framing $\psi$ on $c$ determined by the disk $D$. This can be checked by computing $\la w_2(X), [D_0\cup-D]\ra$ since $\psi_0 \equiv \psi$ (mod $2$) $\Leftrightarrow \la w_2(X), [D_0\cup-D]\ra=0$. 

To find the disk $D$, we shall chase the homotopy class $c$ in the presentation for $\pi_1(X-\Sigma_{K}(m))$ given in Proposition~\ref{P:twistrimfundagp}. Referring to the decomposition~\eqref{cptwistrimsurgery2} for $X-\Sigma_{K}(m)$, we first claim that $c$ bounds a punctured torus in $S^1\times_{\tau^m}(B^3-K_+)$. The twist map $\tau$ along a meridian of $K$ gives a relation $\tau_{*}^m(g)=\mu_K^{-m}g\mu_K^m$ for all $g\in \pi_1(B^3-K_+)$, which is same as $\mu_K^{-1}g\mu_K$ under the assumption on $m$ in our theorem. And since the homotopy class $c$ is $g^{-1}\mu_K^{-1}g\mu_K$ for some $g\in \pi_1(B^3-K_+)$, it is same as $g^{-1}\tau_{*}^m(g)$, which obviously bounds a punctured torus in $S^1\times_{\tau^m}(B^3-K_+)$; see the first picture in~\figref{mappingtorus}.

\fig{150}{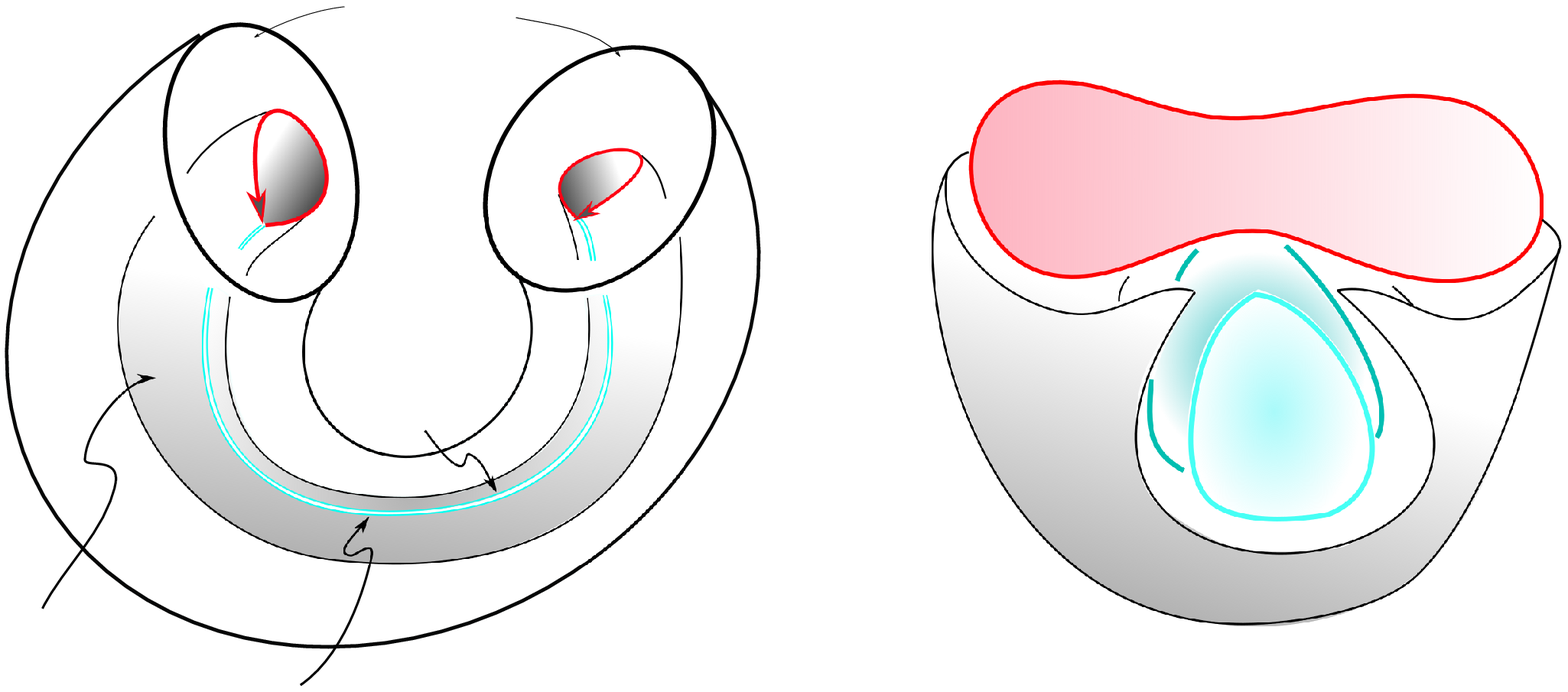}{
\put(-260,150){\footnotesize${(B^3-K_+)}$}
\put(-30,110){\footnotesize$D_0$}
\put(-68,133){\footnotesize$\color{red}c$}
\put(-290,130){\footnotesize$\color{red}g$}
\put(-285,100){\footnotesize$\color{red}*$}
\put(-218,100){\footnotesize$\color{red}*$}
\put(-222,125){\footnotesize$\color{red}\tau_{*}^m(g)$}
\put(-70,55){\footnotesize$D_2$}
\put(-81,90){\footnotesize$-D_2$}
\put(-125,70){\footnotesize$D_1$}
\put(-343,10){\footnotesize$D_1$}
\put(-287,-5){\footnotesize$\delta$}
\put(-260,63){\footnotesize$\delta^{-1}$}
\caption{Disk $D$ in $X-\Sigma_{K}(m)$ with $\p D=c$}
\label{F:mappingtorus}}

Now consider a relation $\delta^{-1}g\delta=\tau_{*}^m(g)$ in $\pi_1(S^1\times_{\tau^m}(B^3-K_+))$ where $\delta$ denotes a generator $[S^1]$ of $\pi_1(S^1\times_{\tau^m}(B^3-K_+),*)$ in~\figref{fundamentalgp}, so a curve representing $\delta^{-1}g^{-1}\delta\tau_{*}^m(g)$ bounds a disk $D_1$ as in the first picture of~\figref{mappingtorus}. And, $\delta$ is nullhomotopic in $X-\Sigma_{K}(m)$ from the presentation~\eqref{mtwistfundgpeq1} in Proposition~\ref{P:twistrimfundagp}, so it bounds a disk $D_2$ in $X-\Sigma-S^1\times (B^3,I)$. Adding this relation to the presentation of $c$, we write $c=\delta^{-1} g^{-1}\delta\tau_{*}^m(g)$, which bounds a disk $D=D_1\cup D_2\cup -D_2$ in $X-\Sigma_{K}(m)$ as depicted in the second picture of~\figref{mappingtorus}. 

It remains to compare the framings $\psi_0$, $\psi$. Note that the element $ [D_0-D]\in H_2(X)$ is same as $[D_0-D_1]\in H_2(S^1 \times_{\tau^m}(B^3,K_+))$ represented by a torus; see the second picture in~\figref{mappingtorus}, which is trivial in $H_2(S^1 \times_{\tau^m}(B^3,K_+))=0$, so does in $H_2(X)$. This shows 
$w_2(X)$ vanishes on this class, from which our result follows.
\end{proof}

\subsection*{\bf II. Attaching a dual handle $H_{*}^2$ of $3$-handle $H^3$}\label{3handle} 
Turning $W$ upside down, the $3$-handle provides adding a $2$-handle $H_{*}^2$ to $\p_{+}\overline W=X\fcs{T_{c}}{T_u}S^1\times S^3\cong (X,\Sigma_{K'}(m))$ in our relative cobordism. As it turns out in~\lemref{dualhandleattachingcircle}, a key point is that the dual handle $H_{*}^2$ is attached to a curve $c'$ at a crossing of $K'$ and its framing, disregarding surface knot $\Sigma_{K'}(m)$ in $X$, is shown to be twisted. But since we're building a relative cobordism from the top, we will find a disk $D$ in $X-\Sigma_{K'}(m)$ bounding $c'$, and then examine the surgery framing relative to this disk $D$. The idea is same as before, so it is basically to find a dual sphere of the attaching sphere of $H^3$ that doesn't intersect with $\Sigma_{K'}(m)$ and determine its framing.

\begin{lemma}\label{L:dual3handle} 
The upper boundary of $\p_{-} W^*\times I\cup H^2_*$ is diffeomorphic to $(X\cs\sst,\Sigma_{K'}(m))$. 
\end{lemma}
\begin{proof} 
Since $H_{*}^2$ is attached along a curve $c'$ at a crossing of $K'$ which is in $S^1\times E(K')\subset (X,\Sigma_{K'}(m))$ as shown in~\lemref{dualhandleattachingcircle}, the curve $c'$ bounds a disk $D_0$ in $S^1\times_{\tau^m} (B^3,K'_+)$ intersecting with $\Sigma_{K'}(m)$ at two points in the same way that $c$ does in~\figref{torusc}. And the surgery framing coming from adding the $2$-handle on $c'$ relative to $D_0$ does not extend over the full disk $D_0$.  Proposition~\ref{P:curveC} shows under our current assumption on $m$ that $c'$ is nullhomotopic in $X-\Sigma_{K'}(m)$, so this surgery gives $(\sss,\emptyset)$ or $(\sss,\emptyset)$ on $(X, \Sigma_{K'}(m))$. The hypothesis on $m$ in $\Sigma_{K'}(m)$ allows one to find a disk $D$ in $X-\Sigma_{K'}(m)$ with the exactly same argument in~\lemref{S2bundleS2} and show that the framing on $c'$ relative to $D$ is equivalent to the one relative to $D_0$ up to (mod $2$). 
\end{proof}

\lemref{S2bundleS2} and~\ref{L:dual3handle} show that $(X\cs\sst,\Sigma_{K}(m))\cong \p_{+} W_{2}=\p_{+}(\p_{-} W^*\times I\cup H^2_*)\cong (X\cs\sst,\Sigma_{K'}(m))$, so it completes the proof of \thmref{onecrosssdiffeotwist}.

\subsection{\bf Rim surgery}\label{S:stabrimsurgeryonecrossing}
Finshel-Stern's rim surgery is the case $m=0$ of $m$-twist rim surgery. Let $\Sigma$ be an embedded surface in a simply-connected $4$-manifold $X$ with $\pi_1(X-\Sigma)=1$. As discussed in~\secref{twistedrimsurgery}, the rim surgery performed in a neighborhood $\nu(\alpha)\cong S^1\times (B^3,I)$ of a curve $\alpha\subset \Sigma$ produces $S^1\times (B^3,K_+)$ i.e. $m=0$ in~\eqref{twistrimsurgery2}.


\begin{theorem}\label{T:onecrosssdiffeorim}
Suppose that two knots $K$, $K'$ in $S^3$ differ by a single crossing change. If $\Sigma_{K}(\varphi)$ and $\Sigma_{K'}(\varphi)$ are surface knots obtained by rim surgery then $(X\cs\sst,\Sigma_{K}(\varphi))$ is pairwise diffeomorphic to $(X\cs\sst,\Sigma_{K'}(\varphi))$.
\end{theorem}

\begin{proof}
Lemma~\ref{L:2handle1} asserts that the level of $W$ after adding a $1$-handle $H^1$ and a $2$-handle $H^2_{\gamma a}$ is $(X\cs S^4,\Sigma_{K}(\varphi))$. At the stage of adding the next $2$-handle $H_{cu}^2$ along $c+u$, the curve $c$ is nullhomotopic in $X-\Sigma_{K}(\varphi)$ because $\pi_1(X-\Sigma_{K}(\varphi))=1$, so the surgery provided from the $2$-handle gives a connected sum of $(\sss,\emptyset)$ or $(\sst,\emptyset)$ on the exterior of the surface. We now need to handle with the framing issue. 

\fig{150}{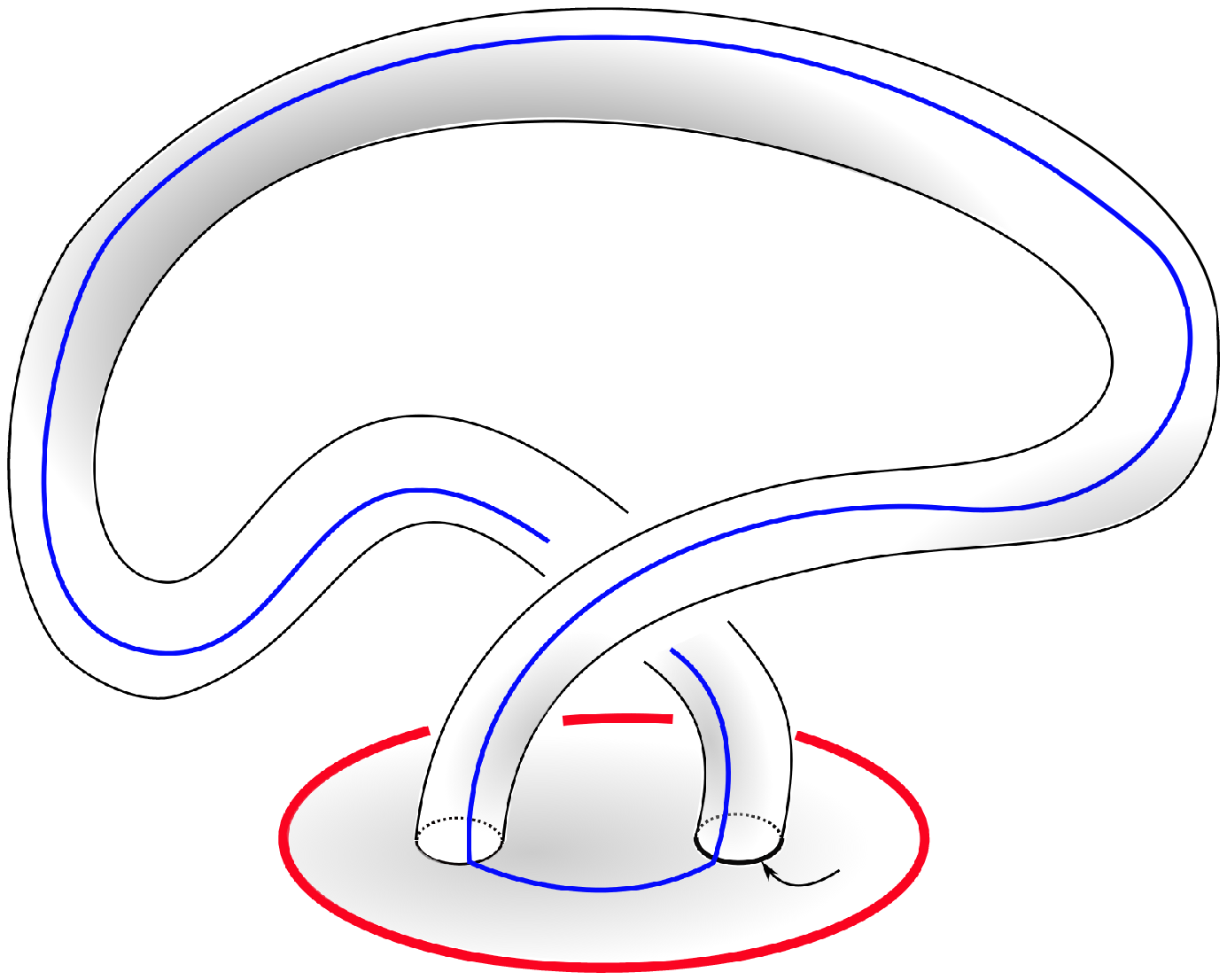}{
\put(-100,5){\footnotesize$\color{blue}g$}
\put(-160,10){\footnotesize$\color{red}c$}
\put(-60,20){\footnotesize$\mu_{K}$}
\caption{Punctured torus $T_*$ in the exterior $E(K)$ of the knot $K$}
\label{F:crosstorus}}

Since $c$ bounds a disk $D_0$ in $S^1\times (B^3,K_+)\subset (X,\Sigma_{K}(\varphi))$ intersecting at two points with $\Sigma_{K}(\varphi)$ as in~\figref{torusc}, this disk specifies the framing $\psi_0$ on the curve $c$ induced from the $2$-handle $H_{cu}^2$, which doesn't extend over $D_0$ by~\lemref{2handlesurgeryframing}. To find a disk $D$ with $\p D=c$ in $X-\Sigma_{K}(\varphi)$, note that $c$ is at a crossing of the knot $K$, so it bounds an obvious punctured torus $T_*$ in $E(K)$ consisting of two generators; a meridian $\mu_K$ of $K$ and some $g\in\pi_1(E(K))$ represented by a blue curve as in~\figref{crosstorus}. Since $\pi_1(X-\Sigma_{K}(\varphi))$ is trivial, the image of $\pi_1(S^1\times E(K))$ is trivial so that the curve $g$ bounds a disk $D_2$ in $X-\Sigma-\nu(R_{\alpha})$ where $R_{\alpha}$ denotes the rim torus given by a curve $\alpha\subset\Sigma$. Cutting $T_*$ along $g$ and filling with two oppositely oriented disks $D_2\cup-D_2$ gives a disk $D$ bounding $c$ in $X-\Sigma_{K}(\varphi)$.

If $\psi$ denotes the framing on $c$ relative to the disk $D$, it readily follows that $\psi$ is equivalent to $\psi_0$ (mod $2$) by showing $\la w_2(X), [D_0\cup-D]\ra=0$. This is because $[D_0\cup-D]=[D_0\cup T_*]$ is represented by a torus in $S^1\times (B^3, K_{+})$, which vanishes in $H_2(S^1\times (B^3, K_{+}))=0$, so does in $H_2(X)$. Thus, the level $\p_{+} W_2$ is diffeomorphic to $(X\cs\sst,\Sigma_{K}(\varphi))$. 

Now turn $W$ upside down and note that as shown in Lemma~\ref{L:dual3handle}, the dual $2$-handle $H^*_2$ of the $3$-handle is attached to the curve $c'$ in $S^1\times E(K')$ which is again nullhomotopic in $X-\Sigma_{K'}(\varphi)$ since $\pi_1(X-\Sigma_{K'}(\varphi))=1$, and so the dual $2$-handle gives $\sss$ or $\sst$ on the exterior of $\Sigma_{K'}(\varphi)$ in $X$. Repeating the same argument in the above with $c'$ in $(X,\Sigma_{K'}(\varphi))$, one can show that the framing of the surgery induced by the $2$-handle $H^*_2$ is twisted and so it proves our theorem.

\end{proof}

\subsection{\bf Annulus rim surgery}\label{S:stabannulusrimsurgeryonecrossing}

Our setting is given as in~\secref{annulusrimsurgery}, and recall that the Finashin's construction is a knot surgery along a torus in a neighborhood $\nu(M)\cong S^1\times (B^3,f)$ to produce $S^1\times (B^3,f_K)=S^1\times (B^3, f)-S^1\times (m_b\times D^2)\cup_{\phi} S^1\times E(K)$ with the gluing $[S^1]\mapsto [S^1]$, $m_b\mapsto \mu_K$, and $\mu_T\mapsto \lambda_K$.
Furthermore, it is not hard to see that $\pi_1(X-\Sigma_K(\phi))$ is preserved when $\pi_1(X-\Sigma)=\Z_d$ by applying the Van Kampen theorem for the decomposition~\eqref{finashin} of $X-\Sigma_K(\phi)$. In this computation, we see that the generators $[S^1]$, $\mu_K$ of $ \pi_1(S^1\times E(K))$ are trivial in $\pi_1(X-\Sigma_K(\phi))$, so the image of $\pi_1(S^1\times E(K))$ is a trivial subgroup of $\pi_1(X-\Sigma_K(\phi))$; see~\cite{finashin},~\cite[Proposition 3.3]{kim-danny:topotriviality} for more details. In this circumstance, the same argument in the rim surgery case works here. 
 
\begin{theorem}\label{T:onecrosssdiffeoannuls}
Suppose that two knots $K$, $K'$ in $S^3$ differ by a single crossing change. If $\Sigma_K(\phi)$ and $\Sigma_{K'}(\phi)$ are surface knots obtained by annulus rim surgery then $(X,\Sigma_K(\phi))\cs(\sst,\emptyset)$ is pairwise diffeomorphic to $(X,\Sigma_{K'}(\phi))\cs(\sst,\emptyset)$.
\end{theorem}
\begin{proof}
We just begin with the $2$-handle $H_{cu}^2$ attached along $c+u$ in $(X\cs S^4,\Sigma_{K}(\varphi))$ from Lemma~\ref{L:2handle1}. Since $c$ is a curve at a crossing of $K$ and the annulus rim surgery is performed on a neighborhood $\nu(M)$, the curve $c$ lies in the resulting manifold $S^1\times (B^3,f_K)=S^1\times (B^3, f)-S^1\times (m_b\times D^2)\cup_{\phi} S^1\times E(K)$. And, it is nullhomotopic in $X-\Sigma_K(\phi)$ since the image $\pi_1(S^1\times E(K))$ is trivial in $\pi_1(X-\Sigma_K(\phi))$ as shown in~\cite{finashin},~\cite[Proposition 3.3]{kim-danny:topotriviality}. So we sketch the exactly same argument in~\thmref{onecrosssdiffeorim}.

There exists an embedded disk $D_0$ bounding $c$ in $S^1\times (B^3,f_K)$ that intersects with $\Sigma_K(\phi)$ at `four points', and the surgery framing relative to the disk $D_0$, coming from the $2$-handle $H_{cu}^2$, doesn't extend over $D_0$ by~\lemref{2handlesurgeryframing}. Since $c$ bounds a punctured torus $T_*$ in $E(K)$ and the image $\pi_1(S^1\times E(K))$ is trivial in $\pi_1(X-\Sigma_{K'}(\phi))$, the argument in~\thmref{onecrosssdiffeorim} gives a way to find another disk $D$ in $X-\Sigma_K(\phi)$ bounding $c$. It readily follows that the framing on $c$ relative to $D$ is also twisted since the homology class $[D_0\cup-D]$ is represented by a torus $[D_0\cup T_*]$ in $H_2(S^1\times (B^3, f_{K}))=0$. So $w_2(X)$ vanishes on this class, from which we have $\p_{+} W_{2}\cong (X\cs\sst,\Sigma_{K}(\varphi))$.

Finally, turn $W$ upside down. By~\lemref{dualhandleattachingcircle}, the dual $2$-handle $H^*_2$ of the $3$-handle is attached along a curve $c'$ in $S^1\times E(K')$, which has a trivial $\pi_1$ in $\pi_1(X-\Sigma_{K'}(\phi))$ so that the attaching circle $c'$ is nullhomotopic in $X-\Sigma_{K'}(\phi)$. One simply proceeds the above argument to show that the boundary $\p_{+}(\p_{-} W^*\times I\cup H^2_*)$ is diffeomorphic to $(X\cs\sst,\Sigma_{K'}(\phi))$, and hence our result follows. 

\end{proof}

\begin{remark}
A crucial point of our argument for the case of rim surgery and annulus rim surgery is that the image $\pi_1(S^1\times E(K))$ is trivial in $\pi_1(X-\Sigma_{K'}(\phi))$. This allows us to exhibit a disk bounding $c$ easily as we first find a punctured torus $T_*$ with $\p T_*=c$ in $E(K)$ and surger out one of two generators of $T_*$ using a disk bounding the circle. But the difference in twist rim surgery is that one cannot proceed this argument, and so we enlarge $S^1\times E(K)$ to the surgered manifold $S^1 \times_{\tau^m}(B^3,K_+)$ of a neighborhood of a curve on $\Sigma$ and use the fact that the image of $\pi_1(S^1 \times_{\tau^m}(B^3-K_+))$ in $\pi_1(X - \Sigma_{K}(m))$ is a cyclic subgroup generated by the meridian of $\Sigma_{K}(m)$ under our hypothesis in Proposition~\ref{P:curveC}. 
\end{remark}
\vskip5pt

\subsection{\bf Stabilization for Rim surgery, Annulus rim surgery, and Twist rim surgery; Proofs of Theorem A, B, and C}\label{S:stabknotsurgery}

\begin{proof}[Proof of \thmref{mainthmA},~\ref{T:mainthmB}, and~\ref{T:mainthmC}]

Suppose that $(X,\Sigma_{K}(\varphi))$ is a pair constructed from rim surgery, twisted rim surgery, and annulus rim surgery on $(X,\Sigma)$, and assume that it preserves its surface knot group under the given hypothesis of Theorems~\ref{T:mainthmA},~\ref{T:mainthmB},~\ref{T:mainthmC}.
For any knot $K$ in $S^3$, there is a sequence of knots $K_1=K$, $K_2$,...., $K_n$, with the unknot $K_n$, by crossing changes i.e. $\pm 1$-Dehn surgery along disjoint $n$-curves $\{c_i\}_{i=1,..n}$ in $E(K)$. So, for each $i$ the pair $(X,\Sigma_{K_{i+1}}(\phi))$ is obtained by a ($\pm 1$)-log transform along a torus $T_{c_i}$ in $(X, \Sigma_{K_{i}}(\phi))$. And at each stage, the surface knot group $\pi_1(X-\Sigma_{K_{i}}(\phi))$ is preserved for each knotting construction so that Theorem~\ref{T:onecrosssdiffeotwist},~\ref{T:onecrosssdiffeorim}, and~\ref{T:onecrosssdiffeoannuls} assert that $\Sigma_{K_{i}}(\phi)$ is equivalent to $\Sigma_{K_{i+1}}(\phi)$ in $X\cs\sst$, and hence we deduce that $(X\cs\sst, \Sigma_{K}(\phi))$ is pairwise diffeomorphic to $(X\cs\sst, \Sigma_{K_{n}}(\phi))$ where $K_n$ is unknot. 

In~\cite[Lemma 2.2]{kim-danny:topotriviality}, it is shown that for the unknot $K_n$, any knot surgery on $X-\nu(\Sigma)$ along a torus $T\subset X-\nu(\Sigma)$ and a gluing $\phi$ with $\phi(\mu_T)=\lambda_{K_n}$ gives a diffeomorphism $(X-\nu(\Sigma))_{K_n}\to X-\nu(\Sigma)$ that is the identity on the boundary. Thus, $(X, \Sigma_{K_{n}}(\phi))\cong (X, \Sigma)$ so that this proves our main theorems. 
\end{proof}


\section{ Stabilization for Cyclic knot surgery}\label{S:stabcylicknotsurgery}

\begin{proof}[Proof of~\thmref{mainthmE}]
We shall follow the standard argument for the $1$-stable equivalence shown in the previous knotting constructions, so the first step is to show that for any two knots $K$ and $K'$ related by a single crossing change, the cyclic knot surgered pairs $(X_{K}(\varphi), \Sigma_{K}(\varphi))$ and $(X_{K'}(\varphi), \Sigma_{K'}(\varphi))$ become pairwise diffeomorphic after connected summing with $(\sst,\emptyset)$. 

We work with the cobordism $W$ constructed in~\secref{cobordism}, and the level of $W$ after adding a $1$-handle $H^1$ and a $2$-handle $H_{\gamma a}^2$ was shown in~\lemref{2handle1} to be diffeomorphic to $(X_{K}(\varphi)\cs S^4,\Sigma_K(\varphi))$. The rest argument about the $2$-handle $H_{cu}$ and $3$-handle $H^3$ is the following.

\lemref{2handlesurgeryframing} shows that attaching the $2$-handle $H_{cu}$ along a curve $c$ in $(X_{K}(\varphi),\Sigma_K(\varphi))$ gives rise to $X_{K}(\varphi)\cs\sst$ for the ambient manifold. This means that its framing on $c$ is determined by an embedded disk $D_0$ in $X_K(\varphi)$ that may intersect with $\Sigma_K(\varphi)$ and the framing does not extend over $D_0$. To consider the framing in the complement of surface knot, we note that $c$ is also nullhomotopic in $X_{K}(\varphi)-\Sigma_{K}(\phi)$ because the element $c$ is $g^{-1}\mu_K^{-1}g\mu_K$ for some $g\in \pi_1(E(K))$ and $\pi_1(X_{K}(\varphi)-\Sigma_{K}(\phi))$ is cyclic. So there is a disk $D$ spanning $c$ in $X_{K}(\varphi)-\Sigma_{K}(\phi)$ that may assume to be embedded as discussed in Remark~\ref{R:embedimmerseddisk}. The framing of surgery along $c$ relative to $D$ can be compared with the one relative to the disk $D_0$ by evaluating $w_2(X_{K}(\varphi))$ on the class $[D_0\cup-D]$, which is zero because $X_{K}(\varphi)$ is spin. Thus, the level $\p_+W_2$ is diffeomorphic to the pair $(X_{K}(\phi)\cs\sst,\Sigma_{K}(\phi))$. 

Turning the $3$-handle upside down, it was shown in~\lemref{dualhandleattachingcircle} that its dual $2$-handle gives a surgery along $c'$ on $(X_{K'}(\phi),\Sigma_{K'}(\phi))$, where the curve $c'$ is at an oppositely oriented crossing of $K'$. Again since $\pi_1(X_{K'}(\phi)-\Sigma_{K'}(\phi))$ is cyclic, $c'$ is nullhomotopic in $X_{K'}(\phi)-\Sigma_{K'}(\phi)$ so that the surgery from the dual $2$-handle yields a connected sum of $(\sss,\emptyset)$ or $(\sst,\emptyset)$ on the boundary.
In the middle level of $W$ between $2$-handles and a $3$-handle, we will have a pairwise diffeomorphism $(X_{K}(\phi)\cs\sst,\Sigma_{K}(\phi))\to (X_{K'}(\phi)\cs\sss,\Sigma_{K'}(\phi))$ or $(X_{K'}(\phi)\cs\sst,\Sigma_{K'}(\phi))$. But it must be the pair $(X_{K'}(\phi)\cs\sst,\Sigma_{K'}(\phi))$ because both ambient manifolds $X_{K}(\phi)$ and $X_{K'}(\phi)$ are spin. 

For the rest argument, the proof in~\secref{stabknotsurgery} applies for the case of cyclic knot surgery with no extra effort. 

\end{proof}

\subsection*{Acknowledgements}
The author would like to thank the American Institute of Mathematics (AIM) for its support. The author is also grateful to Danny Ruberman for helpful comments and to the referee for correcting some errors in the first draft and wonderful suggestions.


\def\cprime{$'$}
\providecommand{\bysame}{\leavevmode\hbox to3em{\hrulefill}\thinspace}
\providecommand{\MRhref}[2]{%
  \href{http://www.ams.org/mathscinet-getitem?mr=#1}{#2}
}
\providecommand{\href}[2]{#2}

\end{document}